\documentclass{amsart}
\usepackage{graphicx,amsfonts,amssymb,amsmath,amsthm,url,longtable}
  \usepackage{color}
\usepackage[usenames,dvipsnames]{xcolor}
\usepackage[hyperfootnotes=false, colorlinks, citecolor=blue, urlcolor=blue, linkcolor=blue	 ]{hyperref}

\theoremstyle{plain} 
\newtheorem{theorem}             {Theorem}  [section]
\newtheorem{lemma}      [theorem]{Lemma}

\newtheorem{proposition}[theorem]{Proposition}

\theoremstyle{definition}
\newtheorem{definition} [theorem]{Definition}

\newtheorem{example}    [theorem]{Example}

\newtheorem{problem}    [theorem]{Problem}

\theoremstyle{remark}
\newtheorem{remark} [theorem]             {Remark}

\allowdisplaybreaks

\def\Art{\operatorname{Art}}
\def\CM{\operatorname{CM}}
\def\AL{\operatorname{AL}}
\def\Gal{\operatorname{Gal}}
\def\Aut{\operatorname{Aut}}

\def\GO{\operatorname{GO}}

\def\tr{\operatorname{tr}}

\def\Hom{\operatorname{Hom}}

\def\vol{\operatorname{vol}}

\def\O{\operatorname{O}}

\def\SO{\operatorname{SO}}
\def\PGL{\operatorname{PGL}}

\def\SL{\operatorname{SL}}
\def\res{\operatorname{res}}
\def\Pic{\operatorname{Pic}}
\def\GL{\operatorname{GL}}
\def\Jac{\operatorname{Jac}}
\def\mathbold{\mathbb}
\renewcommand{\Re}{\mathrm{Re}}
\renewcommand{\Im}{\mathrm{Im}}
\def\eps{\varepsilon}

\def\sec{\S}
\def\secs{\S\S}

\begin{document}

\title{Evaluating modular forms on Shimura curves}

\thanks{The author was supported
  by NSF grant OISE-1064866
  and partially supported by grant SNF-137488
  during the completion of this paper.}

\author{Paul D. Nelson}
\address{EPFL, Station 8, CH-1015 Lausanne, Switzerland}
\email{paul.nelson@epfl.ch}
\begin{abstract}
  Let $f$
  be a newform,
  as specified
  by its Hecke eigenvalues, on a Shimura curve $X$.
  We describe a method for evaluating
  $f$.
  The most interesting case is when $X$ arises as a compact
  quotient
  of the hyperbolic plane, so that classical $q$-expansions are
  not available.
  The method takes the form
  of an explicit, rapidly-convergent formula
  that is
  well-suited
  for numerical computation.
  We apply it to the problem of
  computing
  modular parametrizations of elliptic curves,
  and illustrate with some numerical examples.
\end{abstract}
\maketitle


\section{Introduction}

\subsection{Computing modular parametrizations}
\label{sec:motiv-numer-comp}
Let $E$ be an elliptic curve
over a totally real number field $F$.
By a \emph{modular parametrization} of $E$,
we mean a surjective morphism
\begin{equation}\label{eq:3}
  \Jac X_B \rightarrow E
\end{equation}
of abelian varieties over $F$,
where $X_B$
is the Shimura curve
attached to an order in a quaternion $F$-algebra
$B$ (see \secs\ref{sec:quaternion-algebras2} and
\ref{sec:newf-shim-curv}).
The images under \eqref{eq:3} of certain special divisors
on $X_B$ are called \emph{Heegner points}
(see \sec\ref{sec:setting}),
and provide some of the
most versatile examples of
analytic constructions of solutions
to algebraic equations
(see e.g. \cite{MR803364,GZ86,MR1110395,MR2083209,MR1322717,MR1826411}).




It is of interest to compute (numerically)
the curves $X_B$, the maps \eqref{eq:3},
and the images thereunder of special divisors.
Methods for doing so have been given in many, but not all, cases of interest.
When $X_B$ has very low genus (e.g., genus $0$),
one can work with an explicit algebraic description of
$X_B$ (see \cite{MR1726059,MR2282939});
in general, such a model becomes unmanageable and one must
resort to
analytic means (see \cite[\S 3.2]{darmon-rotger-aws11-notes}).
Methods for computing
Heegner points
using $p$-adic analysis
(see \cite{MR2254648, MR2282936} and references
therein)
apply when $B$ ramifies at \emph{some} finite place,
but there are cases of interest 
in which $B$ ramifies
at \emph{no} finite place,
such as when
$F \neq \mathbold{Q}$ has odd degree
and $E$ is everywhere unramified;
in that setting,
Darmon--Rotger \cite[\S3.2]{darmon-rotger-aws11-notes}
describe
as  ``completely open''
the problem of computing
Heegner points.
Even when $p$-adic methods apply,
it is of independent interest,
and described as an open problem by Greenberg
\cite[p.5]{MR2710023},
to obtain a uniform archimedean-analytic
method for computing Heegner points.

We are motivated by the latter problem of computing
Heegner points,
and more generally modular parametrizations (\ref{eq:3}), using archimedean
analysis.

\subsection{Evaluating modular forms}
\label{sec:relat-probl-comp}
It is classical that the map \eqref{eq:3} is given by
a certain holomorphic differential $f_B(z) \, d z$ on $X_B$,
in the sense that
the composition
\[
\Pic^0 X_B(\mathbold{C}) \cong  \Jac X_B(\mathbold{C}) \rightarrow E(\mathbold{C}) \cong 
\mathbold{C}/\Lambda
\]
of the Abel--Jacobi map
for $X_B$
with the inverse Weierstrass parametrization of $E$
should agree, up to isogeny,
with the linear extension of
\begin{equation*}
  [\tau_1] - [\tau_0]
  \mapsto \int_{\tau_0}^{\tau_1} f_B(z) \, d z.
\end{equation*}
A fundamental
result of Eichler--Shimura (see
\cite[\S7.4]{MR0314766},
\cite[\S1.4.9]{MR1826411},
\cite[\S10.3]{MR860139})
characterizes the weight two newform
$f_B$ (up to a normalizing scalar) by expressing
its Hecke eigenvalues
in terms of the arithmetic
of $E$ over finite fields.

Thus the
problem of computing modular parametrizations
reduces to that of evaluating 
a newform
of known Hecke eigenvalues
on a Shimura curve.
This problem is of independent interest.
For example, near the end of his paper
\emph{Shimura Curve Computations}, Elkies \cite[\S
5.5]{MR1726059} writes:
\begin{quote}
  The reader will note that so far we
  have said nothing about computing with \emph{modular forms} on
  Shimura curves.  Not only is this an intriguing question in its
  own right, but solving it may also allow more efficient
  computation of Shimura curves and the natural maps between them,
  as happens in the classical modular setting.
\end{quote}

We
distinguish three
senses in which one can ``compute''
the space of holomorphic modular forms
of given positive even integral weight on a Shimura curve $X_B$:
\begin{enumerate}
\item[(I)] as an abstract vector space.
\item[(II)] as an abstract Hecke module.
\item[(III)] as a concrete Hecke module realized as a space of differentials on $X_B$.
\end{enumerate}
Solutions to problem (I)
are given by classical dimension
formulas \cite[\S2.6] {MR0314766}
and to problem (II)
by some classical algorithms
involving
Brandt modules \cite{MR579066},
modular symbols \cite{MR2289048},
and their generalizations
(see \cite{MR2291849,MR2467859,MR2721432,MR2772112,
  MR2467860} and references therein).
In this paper,
we address the passage from a known solution of
problem (II)
to one of problem (III).

\subsection{The role of Fourier expansions}
\label{sec:role-four-expans}
If $F = \mathbold{Q}$ is the rational number field,
$B = M_2(\mathbold{Q})$ is the
matrix algebra,
and (for simplicity) $X_{M_2(\mathbold{Q})} = X_0(N)$
is the usual compactification of
$\Gamma_0(N) \backslash \mathbold{H}$
for some natural number $N$,
then $X_{M_2(\mathbold{Q})}$ has the cusp $\infty$ at which $f = f_{M_2(\mathbold{Q})}$ admits
a Fourier expansion
\begin{equation}\label{eq:5}
  f(z) = \sum_{n \in \mathbold{N} } a_f(n) q^n \quad
  \text{ with }q := e(z) := e^{2 \pi i z}, i = \sqrt{-1}
\end{equation}
for some complex coefficients $a_f(n)$,
where $a_f(1)$ is nonzero and explicitly normalized.
The ratios $a_f(n)/a_f(1)$,
which depend multiplicatively on $n$,
are determined explicitly by the Hecke eigenvalues of $f$,
and hence, via the Eichler--Shimura relations,
by computable invariants of $E$.
For example,
if $p$ is a prime not dividing $N$,
then $a_f(p)/a_f(1) = p+1 - \# E(\mathbold{F}_p)$.
Thus one may compute the map $\Jac X_{M_2(\mathbold{Q})}
\rightarrow E$
by summing the series
\[
\int_{\tau_0}^{\tau_1} f(z) \, d z
=
\frac{1}{2 \pi i}\sum_{n \in \mathbold{N} } \frac{a_f(n)}{n} \left( e( n \tau_1) - e( n \tau_0) \right)
\]
obtained
by integrating \eqref{eq:5} term-by-term
(see e.g. \cite{MR1322717, MR2473878}).\footnote{We suppress here discussion
  of certain technicalities, such as the
  choice
  of cusp at which to expand $f$.}

If
$B \not \cong M_2(\mathbold{Q})$,
then $X_B$ has no cusps
(equivalently,
the group $\Gamma$ has no parabolic elements),
so $f_B$ does not admit an expansion
of the type \eqref{eq:5}.
The absence of such an expansion
makes it
more difficult to evaluate
$f_B$;
in fact, we are not aware
of any successful attempts to do so
prior to the work described in this paper.

\subsection{Summary of results}
\label{sec:summary-results}
We will describe
\begin{enumerate}
\item[(A)] a general method for computing the values and line
  integrals
  of modular forms (and their derivatives) on Shimura curves $X_B$ using only
  archimedean analysis (see \sec\ref{sec:method-proofs}), and
\item[(B)] an explicit, ``ready-to-use'' formula,
  obtained by applying the method (A),
  in the simple but already interesting case that $B$ is an
  indefinite 
  quaternion algebra over $\mathbold{Q}$ and $X_B$ is attached
  to an Eichler order of squarefree level (see \sec\ref{sec:explicit-formula}).
\end{enumerate}
We apply the formula (B)
to several numerical examples
in \sec\ref{sec:numerical-examples},
so that \secs\ref{sec:explicit-formula},
\ref{sec:numerical-examples}
and \ref{sec:method-proofs} follow roughly
a ``theorem--application--proof'' paradigm.
Applications of the method (A) to other settings,
such as those arising from Shimura curves
attached to quaternion algebras
over totally real fields,
will be taken up in a future paper.
The method itself is general, and extends
to arbitrary automorphic
quotients attached to unit groups of quaternion algebras over number fields.

\begin{figure}
  \includegraphics{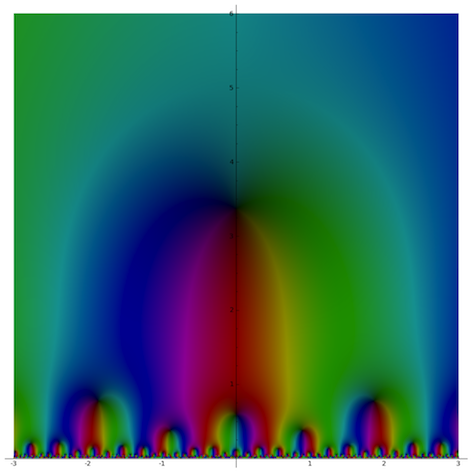}
  \caption{
    The ``first'' modular form on a compact
    Shimura curve:
    $\Im(z)^{4/2} f(z)$
    for $f : \mathbold{H} \rightarrow \mathbold{C}$ in the one-dimensional space
    $S_4(\Gamma_0^6(1))$,
    as depicted by SAGE's \emph{complex\_plot} command
    and evaluated via the method
    of this paper.
    Here $\Gamma_0^6(1) = (\sum_{i=1}^4 \mathbold{Z} e_i) \cap
    \SL_2(\mathbold{R})$,
    where
    $e_1 = [1,0;0,1]$,
    $e_2 = [1/\sqrt{2},1/\sqrt{2};3/\sqrt{2},-1/\sqrt{2}]$,
    $e_3 = [1/{2},1/{2};-3/{2},1/{2}]$,
    and
    $e_4 = [0,\sqrt{2};3\sqrt{2},0]$.
  }
\end{figure}

\subsubsection*{Explicit formula}
To sidestep the notational preliminaries
needed to state formula (B)
precisely,
we illustrate it here in the simplest
nontrivial case.
Let
$\Delta(z) := q
\prod_{n \in \mathbold{N} } (1-q^n)^{24}
= \sum_{n \in \mathbold{N} } a_\Delta(n) q^n
$
be the Ramanujan/discriminant function
on the upper half-plane
$\mathbold{H} := \{ z = x + i y: y > 0 \}$.
It satisfies the functional equation
$\Delta(-1/z) = z^{12} \Delta(z)$
and spans the space of cusp forms
of weight $12$ on $\SL_2(\mathbold{Z})$.
A famous theorem of Deligne
implies that
$|a_\Delta(n)| \leq n^{11/2} \tau(n)$
with $\tau(n)$
the number
of positive divisors of $n$.

One can evaluate $\Delta$ on the positive
imaginary
axis $z = i y$
by summing the series
\begin{equation}\label{eq:21}
\Delta(i y)
= \sum_{n \in \mathbold{Z}} a_\Delta(n) e^{- 2 \pi n y},
\end{equation}
which converges rapidly provided that $y$ is not too small.\footnote{
The latter condition
is not serious,
as  the functional equation
$\Delta(i/ y) = y^{1 2} \Delta(i y)$
allows one to assume that $y \geq 1$.
}
A very special case of Theorem \ref{thm:rough-main-2} is
a new
formula for $\Delta(i
y)$,
which, while visibly ``worse''
than \eqref{eq:21},
has the virtue of extending to modular
forms on compact Shimura curves
for which
analogues of  (\ref{eq:21})
are not available:
\begin{theorem}\label{thm:delta-crazy-sum}
  Let $y_1$ and $y_2$ be positive reals.
  Then
\begin{equation}\label{eq:18}
  (y_1 y_2)^6
  \Delta(i y_1)
  \Delta(i y_2)
  = \sum_{
    \substack{
      a,b,c,d \in \mathbold{Z} \\
      a d - b c > 0
    }
  }
  a_\Delta(a d - b c)
  e^{12 i \theta }
  W_{12} \left( 2 \pi r \right),
\end{equation}
where $r \in \mathbold{R}_+^\times$ and
$\theta \in \mathbold{R} / 2 \pi \mathbold{Z}$
are the polar coordinates defined by\footnote{
  For notational simplicity, we suppress the dependence
  of $r$ and $\theta$ on 
$y_1,y_2,a,b,c,d$.}
\[
r e^{i \theta} =
a \sqrt{y_2/y_1} + d\sqrt{y_1/y_2} 
+
i (b/\sqrt{y_1 y_2} - c \sqrt{y_1 y_2})
\in \mathbold{C}^\times 
\]
and $W_{12}$ is the rapidly-convergent infinite sum of Bessel functions
\[
W_{12}(x)
=
2^{-11} \sum_{n \in \mathbold{N}}
n^{12} ( n x K_{11}(n x) - K_{12}(n x)).
\]
\end{theorem}

It is true, although not completely obvious,
that the RHS of \eqref{eq:18}
converges absolutely, and in fact rapidly;
see Remark \ref{rmk:conv-rapid-sum-alpha}.



The general formula given by
Theorem \ref{thm:rough-main-2}
is no less explicit in principle than
that of
Theorem \ref{thm:delta-crazy-sum}.
It follows that one can
compute modular
parametrizations
by compact Shimura curves
in terms of integrals of expressions essentially of the shape
\eqref{eq:18}
(see \sec\ref{sec:numerical-examples} for numerical examples).

In Theorem \ref{thm:ext-maass-deriv},
we give formulas
for the
values of Shimura--Maass derivatives
of newforms,
which may be used to compute Taylor expansions with respect
to suitable local parameters.
Such formulas should provide an effective tool
for computing equations of Shimura curves,
but we do not pursue such applications here.

The formula \eqref{eq:18} is similar to those
arising from standard methods
for
computing values of $L$-functions,
so one interpretation of the
results
of this paper is that it is no more difficult
to evaluate automorphic forms on
quaternion algebras than to evaluate
$L$-functions of logarithmically-comparable conductor.



\subsubsection*{General method}
We turn to sketching the general method (A),
of which
a more detailed account will be found
in \sec\ref{sec:method-proofs}.
Recall,
from \sec\ref{sec:relat-probl-comp},
that the problem under consideration
is to recover the values
of a newform of given Hecke eigenvalues
on a Shimura curve
(see \sec\ref{sec:background} for definitions).
The abstract possibility of such recovery is the content of the
multiplicity one theorem
(henceforth abbreviated M1)
of Jacquet--Langlands (see \sec\ref{sec:newf-shim-curv}),
so one may view the problem under consideration
as that of giving a computationally effective
realization of M1.
From this perspective, our main tool
suggests itself naturally, as we now explain.


Jacquet--Langlands proved
M1
by establishing
a correspondence
between automorphic forms on $B^\times$ and on $\GL_2$,
identifying
scaling classes of newforms on compact Shimura
curves over $\mathbold{Q}$
with those on congruence covers
of $\SL_2(\mathbold{Z}) \backslash \mathbold{H}$.
Their proof
involved a comparison of instances of the Selberg trace
formula,
and was non-constructive.

Shimizu \cite{MR0333081}, generalizing
earlier work of Eichler, realized the Jacquet--Langlands
correspondence
analytically
as a theta correspondence
between $\GO(B)$ and $\GL_2$.
Since
$\GL_2$ forms admit Fourier expansions,
Shimizu's result reduces our task
to making that theta correspondence computationally
effective.

A key new ingredient, which may
be of independent computational and theoretical interest,
is a technique for computing Petersson inner products
on finite index subgroups of
$\SL_2(\mathbold{Z})$
that does not require knowledge of an explicit fundamental
domain (see \sec\ref{sec:comp-integr-autom}).
We have in mind
the natural generalization
to the setting of
automorphic forms on higher-dimensional
quotients attached to $\GL_2$ over a number field,
where explicit fundamental domains
(and multidimensional oscillating integrals thereon)
seem prohibitively
complicated.



\subsection{Other approaches}
A precursor to the method of this paper
was implemented in Jan 2011, described in the notes
\cite{nelson-appendix-C-prasanna-notes},
and presented at the Arizona Winter School in Mar 2011.
A preprint of Voight and Willis \cite{2012arXiv1205.0045V}
describes another technique, quite different from our own,
for performing computations on
compact Shimura curves.
Inspired by methods of Stark and Hejhal
for computing Fourier expansions of Maass forms on
$\SL_2(\mathbold{Z})$,
the authors compute the Taylor expansion of a modular form
on a compact Shimura curve in a suitable
local parameter by solving for the linear conditions
imposed by the automorphy and Hecke
relations.
Both techniques seem worth developing:
for example,
while the approach described by Voight and Willis
is sufficiently
flexible to extend even to noncongruence or nonarithmetic
Fuchsian groups,
the method described in this paper
seems strong for certain applications, such as
computing modular
parametrizations of elliptic curves,
in which there
are already well-developed methods
for obtaining the Hecke eigenvalues.
We thank John Voight for several lively discussions
of such matters.

\subsection{Acknowledgments}
The problem considered in this paper, as well as our general
strategy, was conceived over the course of several
fruitful conversations with Kartik Prasanna; it is a pleasure to
thank him for numerous helpful discussions that have contributed
substantially to our own understanding.  We thank Andrew Snowden
for helpful discussions during a visit
to the University of Michigan in Jan 2011, at which the first
computations of the sort discussed in this paper were carried
out.
We thank Henri Darmon, Daniel Disegni, Jon Hanke,
Eren Kiral Mehmet,
Victor Rotger,
William Stein, Carlos de Vera, and John Voight
for helpful discussions
and encouragement during the Arizona Winter School 2011,
of which we thank the organizers for their support.
We thank Nahid Walji for his helpful feedback on an earlier draft
of this paper.

\section{Background and notation}
\label{sec:background}

\subsection{Summary}
\label{sec:summary-common-notation}
The remaining subsections (\S\ref{sec:quaternion-algebras2}---\S\ref{sec:jacq-langl-corr-2}) of \S\ref{sec:background}
aim to make precise the summary below,
and
may be skimmed or referred to as necessary.

\subsubsection*{Groups}
For each squarefree integer $D$ possessing an even number
  of prime factors,
  and for each positive integer $N$ coprime to $D$,
  we fix a Fuchsian group
  \[
  \Gamma_0^D(N) = R \cap \SL_2(\mathbold{R}) < \SL_2(\mathbold{R})
  \]
  arising from a choice of Eichler order $R$ of level $N$ in a
  quaternion
  algebra
  $B_{/\mathbold{Q}}$
  of discriminant $D$ (see \S\ref{sec:quaternion-algebras2})
  together with a sequence of inclusions
  \[
  R \subset B \hookrightarrow M_2(\mathbold{R})
  \]
  (see Definition \ref{defn:definition-of-groups-gamma-0-D-N}).
  We take $\Gamma_0^1(N) = \Gamma_0(N)$ for convenience.
\subsubsection*{Newforms}
  For each $\Gamma < \SL_2(\mathbold{R})$ as above
  and each positive even integer $k$, we define
  the set $\mathcal{F}_k(\Gamma)$ of (holomorphic cuspidal) \emph{newforms}
  on $\Gamma$ (see \S\ref{sec:newf-shim-curv}).
\subsubsection*{Jacquet--Langlands correspondents}
Let $\Gamma ' = \Gamma_0^D(N)$ and $\Gamma = \Gamma_0(D N)$.  We
recall the Jacquet-Langlands correspondence as a natural
bijection $\mathcal{F}_k(\Gamma') / \mathbold{C}^\times =
\mathcal{F}_k(\Gamma) / \mathbold{C}^\times$ between scaling
classes of newforms (see \S\ref{sec:jacq-langl-corr-2}).  We say
that a pair of newforms $f' \in \mathcal{F}_k(\Gamma ')$, $f \in
\mathcal{F}_k(\Gamma)$ corresponding thereunder are
\emph{compatibly normalized} if
$f(z)
= \sum_{n \in \mathbold{N}} a_f(n) e(n z)$ with $a_f(1) = 1$ and the Petersson
norms of $f$ and $f'$ with respect to standard hyperbolic measures coincide (see
Definition \ref{defn:compatibly-normalized}).


\subsection{Quaternion algebras
  and Eichler orders}\label{sec:quaternion-algebras2}
We collect here some background on quaternion algebras
and Eichler orders
that will be relevant for what follows.
We refer to \cite{MR580949} for details and proofs.

\subsubsection*{Quaternion algebras over general fields}
Let $F$ be a field.
A \emph{quaternion algebra} $B$ over
$F$
is a simple $F$-algebra with center $F$ and dimension $2^2=4$
over $F$.
A basic example is the algebra $M_2(F)$ of $2 \times 2$
matrices.
A quaternion algebra is called \emph{split},
or said  \emph{to split},
if it is isomorphic to $M_2(F)$
(as an $F$-algebra);
it is called \emph{non-split}, or said
\emph{not to split}, otherwise.
For $a, b \in F^*$,
let $B = (a,b|F)$ denote the
$F$-algebra
$F \oplus F i \oplus F j \oplus F i j$,
where $i, j$ satisfy
the relations
$i^2 = a, j^2 = b$, and $i j = - j i$.
Then $(a,b|F)$ is a quaternion algebra over $F$,
and every quaternion algebra over $F$ is isomorphic
to one of this form
provided that $F$ is not of characteristic $2$.

Let $E$ be an extension field of $F$.
The most relevant examples are quadratic field extensions
of $F$, and completions of $F$ if $F$ is a number field.
Say that
$E$ \emph{splits} $B$, or that $B$ \emph{splits at/over} $E$,
etc., 
if $B \otimes_F E \cong M_2(E)$ splits.
It is known
that $E$ splits $B$ if and only if
there exists an embedding $B \hookrightarrow M_2(E)$
(of $F$-algebras).
Every quaternion algebra over $F$ is split
by some quadratic field extension of $F$.
For example, if $a$ is not a square in $F$, then
$F(\sqrt{a})$ splits
$(a,b|F)$,
and we have an embedding
\[
(a,b|F) \ni
x_0 + x_1 i + x_2 j + x_3 i j \mapsto
\begin{bmatrix}
  x_0 + x_1 \sqrt{a} &  x _2 + x _3 \sqrt{a} \\
  b (x _2 - x _3 \sqrt{a}) & x _0 - x _1 \sqrt{a}
\end{bmatrix}
\in M_2(F(\sqrt{a})).
\]

A quaternion algebra $B$ comes equipped with a
\emph{main involution}
$\iota$ that induces the nontrivial automorphism
of every quadratic field extension of $F$ contained in $B$.
If $B \hookrightarrow M_2(E)$ as above,
then $\iota$
may be obtained by restricting the adjoint map $\left(
  \begin{smallmatrix}
    a &b\\
    c&d
  \end{smallmatrix}
\right)
\mapsto \left(
  \begin{smallmatrix}
    d &-b\\
    -c&a
  \end{smallmatrix}
\right)$.
We denote by
\begin{align*}
  \det(\alpha) &= \alpha \alpha ^{\iota} \in F \\
  \tr(\alpha) &= \alpha + \alpha^{\iota} \in F
\end{align*}
the
(reduced)
\emph{norm} and
(reduced) \emph{trace} on $B$.\footnote{
  For each embedding $B \hookrightarrow M_2(E)$,
  the reduced norm on $B$ is the pullback of the determinant on $M_2(E)$.
  Thus the notation
  ``$\det$'',
  while mildly nonstandard, should
  introduce no confusion.
}
The trace induces an $F$-valued bilinear form
\[\langle b_1, b_2 \rangle = \tr(b_1 b_2^{\iota}) = b_1
b_2^{\iota} + b _1 ^\iota b _2
\]
on $B$, called the \emph{trace pairing}.

\subsubsection*{Quaternion algebras over number fields}
Suppose henceforth
that $F$ is a number field,
and let $|F|$ be its set of places.
If $v \in |F|$,
say that $v$ splits $B$, or that
$B$ splits at/over $v$, if $F_v$ splits $B$.
For example,
if $a,b \in F^\times$, then $(a,b|F)$ splits at $v$
if and only if the Hilbert symbol $(a,b)_v$,
which is defined to be $1$ if $a$ is a norm from
$F_v(\sqrt{b})$ and $-1$ otherwise,
takes the value $1$.

Let $\Sigma_B$ denote the set all $v \in |F|$
that do \emph{not} split $B$.
Class field theory implies the map
$B \mapsto \Sigma_B$
induces a bijection
between
the set of isomorphism classes of quaternion algebras
over $F$ and the set of finite even subsets of $|F|$;
in the case $B = (a,b|F)$, this amounts to the product formula
$\prod_v (a,b)_v = 1$.

The (reduced) \emph{discriminant} $\mathfrak{d}_B$
of $B$ is the
squarefree integral ideal
composed of the primes
at which $B$ does not split.
When $B$ is a quaternion algebra
over the rational number field $F = \mathbold{Q}$,
the discriminant $\mathfrak{d}_B$
is an integral ideal in $\mathbold{Z}$ generated
by a unique positive integer $d_B$, which we shall
also call the discriminant of $B$.

\subsubsection*{Orders}
An \emph{order} $R$ in $B$ is a unital subring of
$B$ that is a lattice, i.e., for which $\dim_{\mathbold{Z}} R = \dim_{\mathbold{Q}} B$.
A \emph{maximal order} is an order that is maximal in $B$ with respect
to inclusion.
For example, if $\mathfrak{o}$ is an order in $F(\sqrt{a})
= F \oplus F i$,
$i^2 = a$,
and $b \in F \cap \mathfrak{o}$,
then $\mathfrak{o} \oplus \mathfrak{o} j$
is an order in $(a,b|F)$;
if $\mathfrak{o}$ is an order in $F$,
then
$M_2(\mathfrak{o})$ is an order in $M_2(F)$;
if $\mathfrak{o}$ is maximal in $F$, then $M_2(\mathfrak{o})$ is
maximal in $M_2(F)$.

An \emph{Eichler order} $R \subset B$
is an intersection of two maximal
orders
in $B$.
Let $\mathfrak{o}$ be the maximal
order in the number field $F$,
and $\mathfrak{o}_v$ its closure in $F_v$
for each finite place $v$ of $F$.
The \emph{level} of an Eichler order $R \subset B$
is the integral ideal $\mathcal{N}$ of $F$,
coprime to $\mathfrak{d}_B$,
such that for each finite place $v$
of $F$ at which $B$ splits, there exists an isomorphism
$B \otimes_F F_v \rightarrow M_2(F_v)$
taking $R \otimes_{\mathfrak{o}} \mathfrak{o}_v$
to $\left(
  \begin{smallmatrix}
    \mathfrak{o}_v &\mathfrak{o}_v \\
    \mathfrak{N} \mathfrak{o}_v & \mathfrak{o}_v
  \end{smallmatrix}
\right)$.
When $F = \mathbold{Q}$
and
$\mathcal{N}$ 
is generated by the positive integer $N$,
we shall refer to $N$ as the level of $R$.

\subsubsection*{Indefinite rational quaternion algebras}
An \emph{indefinite rational quaternion algebra}
is a quaternion algebra over $\mathbold{Q}$ that splits
at the unique real place $\infty$ of $\mathbold{Q}$,
or equivalently, for which there exists an embedding
$B \hookrightarrow M_2(\mathbold{R})$ of $\mathbold{Q}$-algebras.
Class field theory implies that the map $B \mapsto d_B$ induces
a bijection between the set of isomorphism classes of indefinite
rational quaternion algebras
and the set
of positive squarefree integers having an even number of prime
factors.

Let $B$ be an indefinite rational quaternion algebra.
The map sending an Eichler order $R \subset B$ to its level $N
\in \mathbold{N}$
induces a bijection between the $B^\times$-conjugacy classes
of Eichler orders in $B$ and the positive integers
coprime to $d_B$.
For example, if $B = M_2(\mathbold{Q})$,
then $B$ is an indefinite rational
quaternion algebra
of discriminant $d_B = 1$,
and every Eichler order of (positive integral) level $N$ in $B$ is conjugate
to $\left(
  \begin{smallmatrix}
    \mathbold{Z} &\mathbold{Z} \\
    N \mathbold{Z} & \mathbold{Z} 
  \end{smallmatrix}
\right)$.

\begin{definition}
\label{defn:definition-of-groups-gamma-0-D-N}
For each
squarefree positive integer $D$ having an even number
of prime factors
and each positive integer $N$ coprime to $D$,
we choose
a subgroup
\begin{equation}\label{eq:66}
\Gamma_0^{D}(N) < \SL_2(\mathbold{R}),
\end{equation}
as follows.
\begin{itemize}
\item If $D = 1$, we take
$\Gamma_0^{D}(N) = \Gamma_0(N) = \left(
  \begin{smallmatrix}
    \mathbold{Z} &\mathbold{Z} \\
    N \mathbold{Z} &\mathbold{Z} 
  \end{smallmatrix}
\right) \cap \SL_2(\mathbold{Z})
= \left(
  \begin{smallmatrix}
    \mathbold{Z} &\mathbold{Z} \\
    N \mathbold{Z} &\mathbold{Z} 
  \end{smallmatrix}
\right) \cap \SL_2(\mathbold{R})$.
\item If $D \neq 1$, we choose an indefinite rational quaternion
algebra
$B$ of discriminant $D$, a real embedding
$B \hookrightarrow M_2(\mathbold{R})$,
and an Eichler order
$R \subset B$ of level $N$.
We then let $\Gamma_0^{D}(N) = R \cap \SL_2(\mathbold{R})$
be the image of
the group of norm one units in $R$.
\end{itemize}

The group
$\Gamma_0^{D}(N)$ is well-defined
in the sense that its $\SL_2(\mathbold{R})$-conjugacy class
is uniquely determined by 
$D$ and $N$.
The quotient
$\Gamma_0^D(N) \backslash \SL_2(\mathbold{R})$
is compact
if and only if
$D \neq 1$.
\end{definition}

\subsection{Newforms on Shimura curves}\label{sec:newf-shim-curv}
We recall the notion of a (holomorphic) newform on a Shimura curve.

\subsubsection*{Shimura curves}
Let $\mathbold{H} = \{x + i y : y > 0\}$ be the upper half-plane.
The group
$\GL_2(\mathbold{R})^+$
of real $2 \times 2$ matrices with positive determinant
acts on $\mathbold{H}$ by fractional linear
transformations:
$\left(
  \begin{smallmatrix}
    a&b\\
    c&d
  \end{smallmatrix}
\right) z := (a z + b)/(cz + d)$.

Let $\Gamma < \SL_2(\mathbold{R})$ be a lattice,
that is to say a discrete subgroup of finite Haar-covolume.
For example, $\Gamma_0^D(N)$ is a lattice.
More generally, one attaches a lattice $\Gamma$
to each sequence of inclusions
\begin{equation}\label{eq:77}
  R \subset B \hookrightarrow M_2(\mathbold{R}),
\end{equation}
where $B$ is a quaternion algebra with one split real place
over a totally real number field, $B \hookrightarrow M_2(\mathbold{R})$
is a fixed real embedding (unique up to conjugation), and
$R \subset B$ is an order,
by taking $\Gamma = R \cap \SL_2(\mathbold{R})$ the be the image of the group of norm one
units
in $R$.

For each lattice $\Gamma$,
one defines a compact Riemann surface
$X_\Gamma$
by suitably compactifying the quotient
$\Gamma \backslash \mathbold{H}$ (see \cite[\S1.5]{MR0314766}).
By a \emph{Shimura curve},\footnote{This definition is not intended to be exhaustive.}
we mean an $X_\Gamma$
for some
$\Gamma$ arising from a sequence \eqref{eq:77}.
Such curves have canonical models over number fields
(see \sec\ref{sec:setting} and \cite[\S9]{MR0314766}).
We
call
$X_\Gamma$ a \emph{compact Shimura curve}\footnote{This
  terminology
  is a mild misnomer,
  because $X_\Gamma$
  is in all cases a compact Riemann surface.}
if $\Gamma \backslash \mathbold{H}$ is compact.


\subsubsection*{Newforms}
For each $k \in 2 \mathbold{N}$, we recall
the weight $k$ slash operator :
for each function $f:\mathbold{H} \rightarrow \mathbold{C}$
and each $g \in \GL_2(\mathbold{R})^+$,
the function
$f|_k g$ is given by
\[
f|_k g(z) :=
\frac{
  \det(g)^{k/2}
}
{
  (c z + d)^{k}
}
f(g z)
\quad 
 \text{ if }
g = \begin{bmatrix}
  * & * \\
  c & d
\end{bmatrix}.
\]
By an \emph{automorphic function}
of weight $k$ on $\Gamma$,
we mean a smooth function
$f : \mathbold{H} \rightarrow \mathbold{C}$
that satisfies $f|_k \gamma = f$ for all $\gamma \in \Gamma$.
The space $S_k(\Gamma)$ of \emph{cusp forms} of weight $k$ on $\Gamma$
consists of those automorphic functions of weight $k$
that are holomorphic and vanish at the cusps of $\Gamma$, if any.
One knows that $S_k(\Gamma)$ is finite-dimensional,
and that $S_2(\Gamma)$ is isomorphic
to the space of holomorphic $1$-forms on $X_\Gamma$.

Let $\widetilde{\Gamma}$
denote the set of all $\alpha \in \GL_2(\mathbold{R})^+$
for which $\Gamma \alpha \Gamma$ is a finite union of
either left or right $\Gamma$-cosets.
The \emph{Hecke algebra}
$\mathcal{H}(\Gamma) = \mathbold{C} [ \Gamma \backslash
\widetilde{\Gamma } / \Gamma]$,
which consists of formal finite $\mathbold{C}$-linear combinations
of double cosets $\Gamma \alpha \Gamma$ ($\alpha \in
\widetilde{\Gamma}$),
has a natural right action on $S_k(\Gamma)$
that linearly extends $f| \Gamma \alpha \Gamma
= \sum f|_k \alpha_j$ if $\Gamma \alpha \Gamma = \bigsqcup \Gamma
\alpha_j$.
By a \emph{newform} in $S_k(\Gamma)$, we mean an eigenfunction
of $\mathcal{H}(\Gamma)$.\footnote{
  The
  concise definition given here,
  which is more closely aligned with the local theory
  of newforms after Casselman,
  agrees
  with the classical one when $\Gamma = \Gamma_0(N)$.}
Let $\mathcal{F}_k(\Gamma)$ denote the set
of
newforms in $S_k(\Gamma)$.
The set $\mathcal{F}_k(\Gamma)$ is preserved under scaling
by $\mathbold{C}^\times$.

\subsubsection*{Multiplicity one}
For each newform $f \in \mathcal{F}_k(\Gamma)$,
there is a character $\lambda_f : \mathcal{H}(\Gamma) \rightarrow
\mathbold{C}^\times$ of the Hecke algebra
with the property
$f | \varphi = \lambda_f(\varphi) f$
for all $\varphi \in
\mathcal{H}(\Gamma)$.
For general lattices $\Gamma$,
a newform $f$ need not be determined by $\lambda_f$;
in other words,
the one-dimensional irreducible constituents
of the $\mathcal{H}(\Gamma)$-module $S_k(\Gamma)$
need not occur with multiplicity one.
One perspective
is that $\widetilde{\Gamma}$ is generally too
small.

For \emph{arithmetic} lattices $\Gamma$ arising from
Eichler orders in
quaternion algebras,
one knows that $\widetilde{\Gamma} = B^\times \cap
\GL_2(\mathbold{R})^+$ is large,
and (by a theorem of Jacquet--Langlands)
that $\mathcal{H}(\Gamma)$ acts on $S_k(\Gamma)$
with multiplicity one.

In particular,
each newform $f$ on $\Gamma = \Gamma_0^D(N)$
is determined by
its character $\lambda_f$,
the set $\mathcal{F}_k(\Gamma) / \mathbold{C}^\times$
of scaling classes of newforms of given weight is finite,
and the problem of recovering
the values of a newform $f$ from its Hecke data $\lambda_f$
is meaningful.


\subsection{Jacquet--Langlands correspondence}\label{sec:jacq-langl-corr-2}
Let $B$ be an indefinite rational quaternion algebra,
and let $N$ be a positive integer coprime to $d_B$.
Set $\Gamma '  = \Gamma_0^{d_B}(N)$
and $\Gamma = \Gamma_0(d_B N)$.
The Jacquet--Langlands correspondence asserts
(among other things) that there is
a natural bijection of finite sets
\[
\mathcal{F}_k(\Gamma ') / \mathbold{C}^\times
= \mathcal{F}_k(\Gamma) /
\mathbold{C}^\times
\]
between the scaling classes of newforms
on $\Gamma '$ and those on $\Gamma$,
characterized by a certain compatibility between
the actions of $\mathcal{H}(\Gamma ')$
and $\mathcal{H}(\Gamma)$.
This map lifts
non-canonically
to  a bijection
$\mathcal{F}_k(\Gamma ')
\leftrightarrow  \mathcal{F}_k(\Gamma)$.

Briefly, $f_B \in \mathcal{F}_k(\Gamma ')$
corresponds to $f \in \mathcal{F}_k(\Gamma)$
if and only if
for each prime $p$ not dividing $d_B N$,
one has $\lambda_{f_B}(\Gamma' \alpha ' \Gamma ')
= \lambda_{f}(\Gamma \alpha \Gamma)$
for some $\alpha '$ (resp. $\alpha$) of
integral trace and
determinant $p$ in $\widetilde{\Gamma '}$
(resp. $\widetilde{\Gamma}$).


\begin{definition}[Compatibly normalized
  Jacquet--Langlands correspondents]
\label{defn:compatibly-normalized}
Let $f_B \in \mathcal{F}_k(\Gamma ')$ 
and $f \in \mathcal{F}_k(\Gamma)$ be Jacquet--Langlands
correspondents.
Write
\begin{equation}\label{eq:7}
  f(z) = \sum_{n \in \mathbold{N} } a_f(n) e(n z).
\end{equation}
Call $f$ and $f_B$ are \emph{compatibly normalized}
if $a_f(1) = 1$
and the Petersson norms
of $f$ and $f_B$ with respect to
the standard hyperbolic
measures on $\Gamma \backslash \mathbold{H}$
and $\Gamma ' \backslash \mathbold{H}$ coincide,
i.e.,
$  \int _{\Gamma \backslash \mathbold{H} }
  y^k |f(z)|^2 \, \frac{d x
    \, d y}{y^2}
=
  \int _{\Gamma' \backslash \mathbold{H} }
  y^k |f_B(z)|^2 \, \frac{d x
    \, d y}{y^2}.$
\end{definition}

\begin{remark}
\label{rmk:volumes-known}
It is known
(see \cite[eq. 22]{MR0171771} or \cite{MR580949})
that
$\vol(\Gamma \backslash \mathbold{H})
= \frac{\pi }{3}
d_B N
\prod_{p \mid d_B N}
(1 + p^{-1})
$
and
$\vol(\Gamma ' \backslash \mathbold{H})
= \frac{\pi }{3}
d_B N
\prod_{p \mid d_B}
(1 - p^{-1})
\prod_{p \mid N}
(1 + p^{-1}).$
\end{remark}

\begin{remark}
  Let $f$ be as above, with $a_f(1) = 1$.
  The coefficients $a_f(n)$ are determined,
  in an explicit and straightforward manner,
  by $\lambda_f$, hence by $\lambda_{f_B}$.
 They are algebraic integers that may be computed in a variety of software packages (such as
  SAGE \cite{sage2009}).
\end{remark}

\begin{remark}\label{rmk:well-def}
  By compatibly normalizing $f$ and $f_B$,
  we have determined
  the individual values $f_B(z)$ only up to a scalar of magnitude one.
  However, for each pair of points $z_1, z_2 \in \mathbold{H}$,
  the quantity
  \begin{equation}\label{eq:6}
    \overline{f_B(z_1)} f_B(z_2)
  \end{equation}
  \emph{is} well-defined
  in the sense that it is determined by
  $f$ (hence by $\lambda_f$, hence by $\lambda_{f_B}$) and
  by the data $R \subset B \hookrightarrow M_2(\mathbold{R})$
  giving rise to
  the definition of $\Gamma '$.
  Conversely, knowledge of the quantities \eqref{eq:6}
  for all $z_1, z_2 \in \mathbold{H}$
  determines the form $f_B$ up to a scalar of magnitude one.
\end{remark}

\section{Explicit formula}\label{sec:explicit-formula}
In this section we state formulas for the values
(see \sec\ref{sec:main-formula})
and derivatives (see \sec\ref{sec:extens-shim-maass})
of newforms on Shimura curves attached to squarefree
level Eichler orders in indefinite rational quaternion algebras.
\secs\ref{sec:weighted-sums-over}
and \ref{sec:part-dual-latt} introduce notation
that may be referred to as necessary.

\subsection{Main formula
  for the values}\label{sec:main-formula}
Let $B$ be an indefinite rational quaternion algebra,
let $N$ be a positive integer prime to $d_B$,
and
let $R \subset B \hookrightarrow M_2(\mathbold{R})$
be a sequence of
inclusions giving rise to
a lattice $\Gamma ' = \Gamma_0^{d_B}(N)$
as in \secs\ref{sec:quaternion-algebras2}
and \ref{sec:newf-shim-curv}.
Let $\Gamma = \Gamma_0(d_B N)$.
Let $k$ be a positive even integer,
and let $f_B \in \mathcal{F}_k(\Gamma ')$
and $f \in \mathcal{F}_k(\Gamma)$
be compatibly normalized
Jacquet--Langlands correspondents, as in Definition
\ref{defn:compatibly-normalized}
of \sec\ref{sec:jacq-langl-corr-2}.
Let $a_f(n)$ be the $n$th Fourier coefficient of $f$, as in \eqref{eq:7}.
Our immediate goal is to give an explicit formula for
the well-defined quantities $\overline{f_B(z_1)} f_B(z_2)$
(see Remark \ref{rmk:well-def})
in the special case that $N$ is squarefree.

Thus, suppose also that $N$ is squarefree.
To state our formula, we must introduce some notation.
Let $d$ be a divisor of $d_B N$.
A basic consequence
of Atkin--Lehner theory (see \cite{MR0268123}) is
that the Fourier coefficient $a_f(d)$
does not vanish (in fact, that $a_f(d) = \pm d^{k/2-1}$),
hence it makes sense to set
\begin{equation}\label{eq:28}
  c_f(d) := \frac{\mu((d,N))}{d \cdot a_f(d)}
  \quad \left(
  = d  \frac{\mu((d,d_B))}{d}
  \frac{\mu(d)}{d \cdot a_f(d)}
  \right).
\end{equation}
For each point $z=x + i y \in \mathbold{H}$ in the upper half-plane,
define the matrix
\[
\sigma_z := \begin{bmatrix}
  y^{1/2} & x y^{-1/2} \\
  0 & y^{-1/2}
\end{bmatrix}
=
\begin{bmatrix}
  y^{-1/2} &  \\
  & y^{-1/2}
\end{bmatrix}
\begin{bmatrix}
  1 & x \\
  & 1
\end{bmatrix}
\begin{bmatrix}
  y &  \\
  & 1
\end{bmatrix}
\in \SL_2(\mathbold{R}),
\]
which has the property that it sends $i$ to $z$
under the usual action by fractional linear
transformations.

We postpone the introduction of some additional, more involved
notation
$V_k$ and $R^{(d)}$
until \secs\ref{sec:weighted-sums-over}
and
\ref{sec:part-dual-latt}.
Roughly, $V_k(L,t)$ is a rapidly-convergent
weighted sum over a lattice $L <
M_2(\mathbold{R})$,
while $R^{(d)}$ is a lattice in $B
\hookrightarrow  M_2(\mathbold{R})$ that is locally dual to $R$
at primes dividing $d$.
Taking that notation
for granted, our main formula reads:
\begin{theorem}
  \label{thm:rough-main-2}
  For each
  $z_1 = x_1 + i y_1$ and $z_2 = x_2 + i y_2$ in $\mathbold{H}$,
  we have
  \[
  (y_1 y_2)^{k/2}
  \overline{f_B(z_1)}
  f_B(z_2)
  =
  \sum_{n \in \mathbold{N}}
  a_f(n)
  \sum_{d|d_B N}
  c_f(d)
  V_k
  \left(
    \sigma_{z_1}^{-1} R^{(d)} \sigma_{z_2}
    ;
    n/d
  \right).
  \]
\end{theorem}
\begin{proof}  See Appendix \ref{sec:comp-shim-lift}. \end{proof}

\begin{remark}\label{rmk:known-petersson-norm}
  Theorem \ref{thm:rough-main-2} normalizes
  $f_B$ to have known Petersson norm (i.e., equal to that of
  $f$).
  This feature will be convenient in some expected
  applications.
  The Petersson norm of $f$ may be computed in a variety of ways
  (see e.g. Example \ref{example:computing-petersson-norms}).
\end{remark}

\subsection{Weighted sums over lattices}\label{sec:weighted-sums-over}
In this subsection we define $V_k(L,t)$,
which appeared
in  the statement of
Theorem \ref{thm:rough-main-2}.
Let $k$ be a positive even integer,
and let $L$ be a lattice in $M_2(\mathbold{R})$.
For example, $M_2(\mathbold{Z})$ is a lattice
in $M_2(\mathbold{R})$.
For each positive real $x$,
define the infinite sum of Bessel functions
\[
W_k(x)
:=
2^{1-k} \sum_{n \in \mathbold{N}}
n^k ( n x K_{k-1}(n x) - K_k(n x)).
\]
We refer the reader interested in how
$W_k$ arises in nature
to Appendix \ref{sec:comp-shim-lift},
and especially (\ref{eq:35}).

There is a unique $\mathbold{R}$-linear map
$\iota : \mathbold{C} \rightarrow M_2(\mathbold{R})$
sending $i$ to $\left(
  \begin{smallmatrix}
    0&1\\
    -1&0
  \end{smallmatrix}
\right)$.
Let $\eps = \left(
  \begin{smallmatrix}
    1&\\
    &-1
  \end{smallmatrix}
\right)$.  Then $\eps^2 = 1$ and $\eps \iota (z) \eps = \iota
(\bar{z})$ for all $z \in \mathbold{C}$.  We have an orthogonal
(with respect to the trace pairing) decomposition $M_2(\mathbold{R}) =
\iota(\mathbold{C}) \oplus \eps \iota(\mathbold{C})$,
which gives rise to
maps $X, Y : M_2(\mathbold{R})
\rightarrow \mathbold{C}$
by requiring that
$\alpha = \iota (X(\alpha)) + \eps \iota (Y(\alpha))$
for all $\alpha \in M_2(\mathbold{R})$.
As functions on $M_2(\mathbold{R})$,
we have $\det = |X|^2 - |Y|^2$.
Set $P := |X|^2 + |Y|^2$.
Explicitly,
\begin{equation*}
  X(\alpha)
  = \frac{
    a+d + i(b-c)
  }{2},
  \,\,
  Y
(\alpha)
= \frac{a - d + i (b + c )}{2},
\,\,
P(\alpha) =
\frac{a ^2 + b ^2 + c
    ^2 +d ^2}{2}
\end{equation*}
for $\alpha = \left(
  \begin{smallmatrix}
    a &b\\
    c&d
  \end{smallmatrix}
\right) \in M_2(\mathbold{R})$.

Finally, for each positive real $t$, set
\begin{equation}\label{eq:17}
  V_k(L,t)
  :=
  \sum _{
    \substack{
      \alpha \in L \\
      \det(\alpha) = t
    }
  }
  e^{i k \arg X(\alpha)}
  W_k(4 \pi |X(\alpha)|).
\end{equation}
Here
$e^{i k \arg X(\alpha)}
= (X(\alpha)/|X(\alpha)|)^k$.
\begin{remark}
  \label{rmk:conv-rapid-sum-alpha}
  To see that the sum
  defining $V_k(L,t)$ converges rapidly,
  note that $W_k$ is bounded
  and decays exponentially,
  and that for each $x \geq 1$,
  there are at most
  $O_{L}(x^4)$ elements
  $\alpha \in L$ for which
  $\det(\alpha) > 0$ and
  $|X(\alpha)| \leq x$.
  The latter claim follows from
  the inequality
  $2|X(\alpha)|^2 \geq 2|X(\alpha)|^2 - \det(\alpha)
  = P(\alpha)$.

  When computing \eqref{eq:17} numerically,
  one restricts the sum over $\alpha$ to those
  elements of $L$ for which the positive
  definite quadratic form $P$
  takes values smaller than
  some fixed large cutoff.
\end{remark}
\begin{example}
  \label{ex:m2z-weighted-sum-spelled-out}
  If $L = M_2(\mathbold{Z})$,
  then for each positive integer $n$ we have
  \[
  V_k(L,n)
  = \sum_{
    \substack{
      a,b,c,d \in \mathbold{Z} \\
      a d - b c = n
    }
  }
  e^{ i k \theta }
  W_{k} \left( 2 \pi r \right),
  \quad \text{ where }
r e^{i \theta} = a + d
  + i (b - c)
  \in \mathbold{C}^\times.
  \]
\end{example}
\begin{remark}
  Unlike in the case of classical (parabolic)
  Fourier expansions, the special functions $W_k$ that
  arise in the definition \eqref{eq:17},
  and hence in
  Theorem \ref{thm:rough-main-2},
  are not representation-theoretically significant.
  Indeed, they are somewhat arbitrary; other choices are
  possible,
  although
  ours is convenient owing to its
  direct expression in terms of the classical $K$-Bessel
  function (see
  \sec\ref{sec:comp-integr-autom} and
  Appendix \ref{sec:comp-shim-lift}).
  They are instead analogous
  to the test functions that arise
  in the approximate functional equation method
  for computing values of $L$-functions outside their domain of
  absolute
  convergence (see \cite[\S 5.2]{MR2061214}).
  In fact, these is nonempty overlap between
  that method and our own (see
  Example \ref{example:computing-petersson-norms}), although
  neither is a special case of the other.
\end{remark}

\subsubsection*{A generalization relevant for computing Shimura--Maass derivatives}
The remainder
of this subsection
will be needed
only in the statement of Theorem \ref{thm:ext-maass-deriv}.
Define derivations
$\delta_1,\delta_2$ on the formal power series
ring
$\mathbold{Z}[[x,\bar{x},y,
\bar{y}]]$
via the table
\begin{center}
  \begin{tabular}{|l|l|l|l|l|}
    \hline
    $j$  & $\delta_j x$ & $\delta_j \bar{x}$ & $\delta_j y$ & $\delta_j \bar{y}$
    \\ \hline \hline
    $1$ & $0$ & $- \bar{y}$ & $- x$ & $0$
    \\ \hline
    $2$ & $0$ & $y$ & $0$ & $x$
    \\ \hline
  \end{tabular}
\end{center}
For nonnegative integers
$j_1$ and $j_2$,
define a polynomial $p_{j_1,j_2}$
via
\begin{equation*}
  p_{j_1,j_2} =
  x^{-j_1 + j_2} 
  e^{(x \bar{x} + y \bar{y})/2}
  \delta_1^{j_1}
  \delta_2^{j_2}
  e^{-(x \bar{x} + y \bar{y})/2}
\end{equation*}
The definition should be regarded as analogous
to that of the
Hermite polynomials
$H_n(x)
= (- 1) ^n e ^{x ^2 }
\frac{d ^n }{d x ^n } e ^{- x ^2 }.$
An induction shows that $p_{j_1,j_2}$
has integral coefficients and
is independent of $x$ and $\bar{x}$,
so that we may write
\[
p_{j_1,j_2}=
\sum_{e_1,e_2 \in \mathbold{Z}_{\geq 0}}
C_{e_1,e_2}^{j_1,j_2} y^{e_1} \bar{y}^{e_2}
\]
for some integers
$C_{e_1,e_2}^{j_1,j_2}$.
Now set
\begin{equation}\label{eq:19}
  V_k^{j_1,j_2}(L,t)
  := \sum_{e_1,e_2 \in \mathbold{Z}_{\geq 0}}
  (4 \pi )^{
    (j _1 + j _2 + e _1 + e _2 )/2
  }
  C_{e_1,e_2}^{j_1,j_2} V_{k;e_1,e_2}^{j_1,j_2}(L,t),
\end{equation}
where
\[
V_{k;e_1,e_2}^{j_1,j_2}(L,t)
:=
\sum_{\substack{
    \alpha \in L \\
    \det(\alpha) =t
  }
}
e^{i (k + j_1 + j_2) \arg X(\alpha)}
Y(\alpha)^{e_1}
\overline{Y}(\alpha)^{e_2}
W_{k + \frac{j_1 + j_2 + e_1 + e_2}{2}} \left( 4 \pi \lvert
  X(\alpha) \rvert \right).
\]

\begin{remark}
  One can express $p_{j_1,j_2}$
  in terms of
  Laguerre
  polynomials
  \begin{equation}\label{eq:9}
    L_{\ell}^{(\alpha)}(x)
    = \sum _{i = 0}^{\ell}
    \frac{(-x)^i}{i!}
    \binom{\ell+\alpha }{\ell-i}
  \end{equation}
  as
  $p_{j_1,j_2}(y,\bar{y})
  = j ! (- y)^{j_2-j} \bar{y}^{j_1-j}
  L_{j}^{(J-j)} \left( y \bar{y} \right)$
  with
  $j = \min(j_1,j_2)$
  and $J = \max(j_1,j_2)$.
  This identity is not essential for our purposes,
  so we omit the proof.
\end{remark}
\begin{example}
  ~
 $p_{0,0} = 1$, hence
    $V_{k}^{0,0}(L,t) = V_k(L,t)$.
   $p_{0,j_2} = (- y)^{j_2}$. $p_{j_1,0} =  \bar{y}^{j_1}$.
  $p_{1,j_2} = (-1)^{j_2} y^{j_2-1} (y \bar{y} - j_2)$.
    $p_{j_1,1} = -\bar{y}^{j_1-1}(y \bar{y} - j_1)$.
   $p_{j,j}
    = j! L_j^{(0)}(y \bar{y})
    = \left. e^t \frac{d^j}{d t^j}
      (t^j e^{-t}) \right|_{t=y \bar{y}}$.
\end{example}

\subsection{Locally dual lattices}
\label{sec:part-dual-latt}
Let $B$ be an indefinite rational quaternion algebra,
and let $R \subset B$ be an Eichler order of level $N$.
We define here
for each divisor $d$
of $d_B N$
the auxiliary lattice $R^{(d)}$
attached to $R$ that appeared in the statement
of Theorem \ref{thm:rough-main-2} in \S\ref{sec:extens-shim-maass}.

Let $R^*$ be the lattice dual to $R$ with respect
to the trace pairing.
Thus, $R$ is the set of all $\alpha \in B$ for which $\langle
\alpha,r \rangle
\in \mathbold{Z}$ for all $r \in R$.
We have $R^* \supset R$.

It is a general fact that lattices in vector spaces over
$\mathbold{Q}$
may be characterized locally.
In our context, this means more precisely the following:
Let $p$ be a prime number,
and set
$B_p = B \otimes_\mathbold{Z}  \mathbold{Z}_p = B \otimes_\mathbold{Q}
\mathbold{Q}_p$.
For each lattice $L < B$
(such as $L = R$ or $L = R^*$),
set $L_p = L \otimes_\mathbold{Z}  \mathbold{Z}_p$.
Then $L_p$ is lattice in $B_p$.
If $L, L'$ are lattices in $B$
that satisfy $L_p = L'_p$ for all primes $p$,
then $L = L'$.

We may characterize the dual lattice $R^*$ as follows:
\begin{itemize}
\item If $p \nmid d_B N$,
  then $R_p = R_p^*$.
\item If $p \mid d_B$,
  then $B_p$ is a discrete valuation ring,
  $R_p$ is the unique maximal order in $B_p$,
  and $R_p^*$ is the inverse of its unique maximal ideal.
\item If $p \mid N$,
  there exists an isomorphism
  $\iota : B_p \rightarrow M_2(\mathbold{Q}_p)$
  taking $R_p$ to
  the Eichler order
  $\left(
    \begin{smallmatrix}
      \mathbold{Z}_p&\mathbold{Z}_p\\
      N \mathbold{Z}_p&\mathbold{Z}_p
    \end{smallmatrix}
  \right)$.
  We then have
  $\iota (R_p^*)
  = 
  \left(
    \begin{smallmatrix}
    \mathbold{Z}_p& N ^{-1} \mathbold{Z}_p\\
    \mathbold{Z}_p&\mathbold{Z}_p
  \end{smallmatrix}
  \right)$.
\end{itemize}

For each divisor $d$ of $d_B N$,
there is a unique intermediate
lattice $R^* \supset R^{(d)} \supset R$ for which
\[
R^{(d)}_p
= \begin{cases}
  R_p
  &
  \text{ if } p
  \nmid d,\\
  R^*_p &
  \text{ if }
  p \mid d
\end{cases}
\]
as lattices in $B_p$.
We have $R^{(1)} = R$ and $R^{(d_B N)} = R^*$,
while
$R^{(d)}$ is indeed locally dual to $R$ at the prime divisors of
$d$.

\subsection{Extension to Shimura--Maass derivatives}
\label{sec:extens-shim-maass}
A generalization of Theorem \ref{thm:rough-main-2} gives a formula for
arbitrary
Shimura--Maass 
derivatives of a newform $f_B$.
These determine the geodesic polar expansion
of $f_B$ about an arbitrary point,
which is useful for rapidly evaluating
$f_B$ at a large number of points.

Let $k$ be a positive even integer,
and $\Gamma < \SL_2(\mathbold{R})$ a lattice.
Let
\[
\delta_k :=
y ^{- k/2 - 1 } 
\left( 
  2 i y \frac{\partial }{\partial z}
  + \frac{k}{2}
\right)
y ^{k/ 2}
=
(-4 \pi ) \cdot 
\frac{1}{2 \pi i }
\left( \frac{\partial }{\partial z}
  + \frac{k }{2 i y}\right)
\]
be the Shimura--Maass raising operator sending (not
necesssarily holomorphic) automorphic functions
on $\Gamma$
of weight $k$
to
those of weight $k+2$.
This agrees with the action
by the element
$\frac{1}{2}\left(
  \begin{smallmatrix}
    1&i\\
    i&-1
  \end{smallmatrix}
\right)$
of the complexified Lie algebra of $\SL_2(\mathbold{R})$
under the usual identification of automorphic functions
$f$
of weight $k'$ ($k' = k$ or $k' = k+2$) on $\Gamma$
with functions
$F(g) = (f|_{k'} g)(i)$
on $\SL_2(\mathbold{R})$
satisfying
$F
\left( \gamma g
  \left(
    \begin{smallmatrix}
      \cos \theta & \sin \theta \\
      - \sin \theta & \cos \theta 
    \end{smallmatrix}
  \right)
\right)
= F(g) e^{i k' \theta}$
for all $\gamma \in \Gamma, \theta \in \mathbold{R} / 2 \pi \mathbold{Z}$.
For notational simplicity,
we henceforth omit the subscript $k$ from $\delta_k$
and write $\delta^{\ell}$ for the $\ell$-fold
composition $\delta_{k + 2l - 2} \circ \dotsb
\circ \delta_{k+2} \circ \delta_k$.

\begin{example}
  If $\Gamma$ contains $\left(
    \begin{smallmatrix}
      1&1\\
      &1
    \end{smallmatrix}
  \right)$
  and
  $f \in S_k(\Gamma)$
  has the Fourier expansion
  $f(z) = \sum_{n \in \mathbold{N}} a_n q^n$,
  then its first couple of Shimura--Maass derivatives are
  $\delta f(z)
    = \sum_{n \in \mathbold{N}} \left( \frac{k}{y} - 4 \pi n \right)
    a_n q^n$
    and
    $\delta^2 f(z)
    = \sum_{n \in \mathbold{N}} \left( \frac{k(k+1)}{y^2} -
      \frac{2(k+1)(4 \pi n)}{y}
      + (4 \pi n)^2
    \right)
    a_n q^n$.
  For $\ell \in \mathbold{Z}_{\geq 0}$,
  an induction confirms that
  $\delta^{\ell} f$
  is the weight $k + 2 \ell$
  automorphic function with Fourier expansion
  $\delta^{\ell} f(z)
  =
  y^{-\ell} \ell!
  \sum_{n \in \mathbold{N}}
  L_{\ell}^{(k-1)}(4 \pi n y)
  a_n q^n$,
  where $L_{\ell}^{(k-1)}$
  is a Laguerre polynomial as in \eqref{eq:9}.
\end{example}

\begin{theorem}\label{thm:ext-maass-deriv}
  With notation and assumptions
  as in Theorem \ref{thm:rough-main-2},
  and $V_k^{j_1,j_2}$ as in \eqref{eq:19},  
  we have
  \begin{equation*}
    y_1^{\frac{k}{2} + j_1}
    y_2^{\frac{k}{2} + j_2}
    \overline{\delta^{j_1} f_B(z_1)}
    \delta^{j_2} f_B(z_2)
    =
    \sum_{n \in \mathbold{N} } a_f(n)
    \sum_{d|d_B N}
    c_f(d)
    V_k^{j_1,j_2}
    \left(
      \sigma_{z_1}^{-1} R^{(d)} \sigma_{z_2}
      ;
      \frac{n}{d}
    \right).
  \end{equation*}
\end{theorem}
\begin{proof}  See Appendix \ref{sec:comp-shim-lift}. \end{proof}

\subsection{Complements}

\subsubsection*{Specialization to CM points}
The following brief remarks
are not necessary for general comprehension of the paper.
Let $K$ be an imaginary quadratic field.
Suppose that $z = z_1 = z_2 \in \mathbold{H}$
is the fixed point of some
embedding $K \hookrightarrow B$.
Then
the lattices
$\sigma_{z}^{-1} R^{(d)} \sigma_{z}$
arising in Theorems \ref{thm:rough-main-2} and \ref{thm:ext-maass-deriv}
decompose
as ``sums of
translates of
products''
of quadratic lattices tailored
to the orthogonal decomposition
$B = K \oplus
K^\perp$ afforded by the trace pairing
(see \cite{prasanna-appendix-B-prasanna-notes} for a 
more precise discussion of related issues),
leading to a pleasant formula for
$\overline{\delta^{j_1} f_B(z)} {\delta^{j_2} f_B(z)}$
involving sums
of the shape
\[
\mathop{\sum \sum}_{
  \substack{
    m_1, m_2 \geq 0 \\
    \ell_1 m_1 > \ell_2 m_2
  }
}
a_f(\ell_1 m_1 - \ell_2 m_2)
r_1(m_1)
r_2(m_2)
W_{k + \frac{j_1 + j_2 + e_1 + e_2}{2}}
(4 \pi \sqrt{m_1}).
\]
Here $\ell_1$ and $\ell_2$ are positive integers,
while $r_1(m_1)$ and $r_2(m_2)$ are the Fourier coefficients
of certain imaginary quadratic theta series given by
\[
r_1(m_1)
=
\sum_{\substack{\alpha \in \mathfrak{a}_1 \\
    |\alpha|^2 = m_1
  }
}
\left( \frac{\alpha }{|\alpha|} \right)^{k + j_1 + j_2},
\quad 
r_2(m_2)
=
\sum_{\substack{\alpha \in \mathfrak{a}_2 \\
    |\alpha|^2 = m_2
  }
}
\alpha^{e_1} \overline{\alpha }^{e_2}
\]
for some (translates of) lattices $\mathfrak{a}_1, \mathfrak{a}_2$ in $K$.
We do not state any precise forms of such identities here,
but remark that we have implemented their derivation
algorithmically.
It would be interesting to compare what
one gets in this way when
$j_1 = j_2$
with what one would get
via an explicit form of Waldspurger's formula
and an approximate functional equation.

In this way, one simplifies the problem of evaluating
$V_k(L;t)$ from one of enumerating
short vectors in a rank
four
quadratic lattice over $\mathbb{Q}$
to one of enumerating short vectors in a pair of
rank one
Hermitian lattices over $K$.

\subsubsection*{Remarks on assumptions}
We have restricted $N$ to be squarefree in our formulas
only because our method of choice requires knowledge of the
Fourier expansion of $f$ at each cusp of $\Gamma_0(d_B N)$
(see \sec\ref{sec:method-proofs}).
The problem of writing down such expansions for forms of non-squarefree level, while not altogether trivial,
seems orthogonal to the primary purposes of this paper
(see e.g. \cite[\S1.9]{PDN-AP-AS-que} for further discussion).
For weights $k > 2$, one can use the holomorphic projection
method described in \sec\ref{sec:some-cand-appr}
to obtain formulas valid for general levels
involving the Fourier expansion of $f$ only at the cusp $\infty$.
See Remark \ref{rmk:holom-proj-wt-2}
for some musings on whether
this method can be extended to the crucial case $k=2$.

\section{Numerical examples}\label{sec:numerical-examples}
In \cite{nelson-appendix-C-prasanna-notes}, we gave several examples of
computations of the values and derivatives of modular forms on
compact Shimura curves
using a precursor of the method presented thus far.
For example, we checked the rationality of their absolute values
at CM points up to powers
of known real periods.

Shortly after \cite{nelson-appendix-C-prasanna-notes}
was presented, we improved the method
to its present form and
worked out some numerical examples
along the lines of the motivation discussed
in \sec\ref{sec:motiv-numer-comp}.
In this section, we record several such examples.

In more detail, this section is organized as follows.
In \S\ref{sec:setting}, we recall
how given a
(modular) elliptic curve $E_{/\mathbb{Q}}$
and
an imaginary quadratic discriminant $D$
satisfying certain compatibility conditions,
together with an auxiliary Atkin-Lehner
operator $\tau$,
one can attach
a point $P(D, \tau) \in E(H_D)$
(with $H_D$ the ring class field attached to $D$)
which is well-defined
up to automorphisms of $E$ and Galois conjugacy.
We formulate
a precise sense in which it is meaningful to compute such
points (Problem \ref{prob:compute-stuff}).
In \S\ref{sec:numerical-examples-method},
we recall how such computation
reduces to the evaluation of modular forms on Shimura curves.
In \S\ref{sec:numerical-examples-results},
we record a small sample of numerical results
obtained using the formulas stated in \S\ref{sec:explicit-formula}.

\subsection{Setting}\label{sec:setting}
Let $B$ be an indefinite rational quaternion algebra,
let $N$ be a positive integer coprime to $d_B$,
and let $R \subset B \hookrightarrow M_2(\mathbold{R})$
be an Eichler order and fixed real embedding
giving rise to $\Gamma ' = \Gamma_0^{d_B}(N) = R \cap
\SL_2(\mathbold{R})$
as in \secs\ref{sec:quaternion-algebras2}
and \ref{sec:newf-shim-curv}.
Let $X_B := X_0^{d_B}(N)$ be the Shimura curve
attached to $\Gamma '$ as in \sec\ref{sec:newf-shim-curv}.

\subsubsection*{CM points}
Let $D$ be a negative quadratic discriminant,
so that $D = d_K c^2$
for some $c \in \mathbold{N}$
and some imaginary quadratic field  $K$ of discriminant $d_K < 0$.
There is a unique (up to isomorphism)
quadratic order $\mathcal{O} = \mathcal{O}_D$
of discriminant $D$,
constructed as $\mathcal{O} = \mathbold{Z} [c (d_K+\sqrt{d_K})/2]$.
Let $H_D/K$ be the
ring class
extension of $K$ attached to $\mathcal{O}$,
so that class field theory gives a
natural isomorphism $\Art : \Pic(\mathcal{O}) \rightarrow
\Gal(H_D/K)$.
Let $h(D) = \# \Pic(\mathcal{O})
= [H_D: K]$ be the class number of $D$.

Call an embedding $\mathcal{O} \hookrightarrow R$
\emph{optimal}
if its extension to $K \hookrightarrow
B$
satisfies $K \cap R = \mathcal{O}$.
Let $\CM_D(X_B)$ denote the set of images in $X_B$
of fixed points in $\mathbold{H}$
for optimal embeddings
$\mathcal{O}  \hookrightarrow R$
(recall that $R \hookrightarrow
M_2(\mathbold{R})$ is fixed).
There is a faithful action
of $\Pic(\mathcal{O})$ on
the set of optimal embeddings $\mathcal{O} \hookrightarrow R$,
which descends to a faithful action on
$\CM_D(X_B)$ that we denote by
$\zeta \cdot \mathfrak{a}$ for $\zeta \in \CM_D(X_B)$ and
$\mathfrak{a} \in \Pic(\mathcal{O})$:
if $\zeta$ is fixed by an optimal embedding $j : \mathcal{O} \hookrightarrow R$,
and $\gamma \in B \cap \GL_2(\mathbold{R})^+$
is chosen to satisfy $\gamma j(\mathfrak{a}) R = R$,
then $\zeta \cdot \mathfrak{a} = \gamma \zeta$.

\subsubsection*{Atkin--Lehner operators}
Let $R_+ = B \cap \GL_2(\mathbold{R})^+$
be the monoid of positive norm elements in $R$.
Call each $\tau \in R_+$
that normalizes $\Gamma '$
an \emph{Atkin--Lehner operator}.
The image in $\Aut(\Gamma' \backslash \mathbold{H})$
of the Atkin--Lehner operators
is a group that we denote by $\AL(\Gamma ')$.
We have
$\AL(\Gamma ') \cong (\mathbold{Z}/2)^{\omega(d_B N)}$, where $\omega (d_B N)$ is the number of prime factors
of $d_B N$,
with representatives $\tau_d  = \prod_{p | d} \tau_p$
indexed
by the positive squarefree divisors
$d$ of $d_B N$, where:
\begin{itemize}
\item For $p \mid d_B$,  $\tau_p \in R_+$ is an arbitrary
  element
  of norm $p$.
\item For $p^\alpha \mid \mid N$, $\tau_p \in R_+$
  is an element of norm $p^\alpha$
  that maps
  to $\left(
    \begin{smallmatrix}
      0&-1\\
      p^\alpha &0
    \end{smallmatrix}
  \right)$
  under a surjection
  $R \rightarrow M_2(\mathbold{Z}/p^{\beta})$
  taking
  $R$
  to $\left(
    \begin{smallmatrix}
      \mathbold{Z}/p^{\beta} & \mathbold{Z}/p^{\beta}\\
      p^\alpha  \mathbold{Z}/p^{\beta} & \mathbold{Z}/p^{\beta}
    \end{smallmatrix}
  \right)$
  for some $\beta > \alpha $.
\end{itemize}

\subsubsection*{Canonical models}
We regard $X_B$ as an algebraic curve over $\mathbold{Q}$ with
respect to its canonical model as in \cite[\S9]{MR0314766},
characterized as follows:
for each negative quadratic discriminant $D$
and each $\zeta \in \CM_D(X_B)$,
one has $\zeta \in X_B(H_D)$;
moreover, for each $\mathfrak{a} \in \Pic(\mathcal{O}_D)$,
one has $\zeta \cdot \mathfrak{a} = \zeta^{\Art(\mathfrak{a})}$.
The group
$\AL(\Gamma ')$ acts on $X_B$ via maps defined
over $\mathbold{Q}$,
hence  preserving $\CM_D(X_B)$
and commuting with $\Gal(H_D/K)$.
The group $\AL(\Gamma ') \times \Gal(H_D/K)$ acts
transitively on $\CM_D(X_B)$.

\subsubsection*{Heegner points}
Let $E$ be an elliptic curve over $\mathbold{Q}$
of conductor $d_B N$,
and let $\mathcal{O} = \mathcal{O}_D \subset K = \mathbold{Q}
(\sqrt{D})$
be an imaginary
quadratic order of discriminant $D$.
We assume that
$N$ is squarefree,
that $c$ is prime to $N$,
and that $\CM_D(X_B) \neq \emptyset$.
In view of the former assumptions,
the latter holds if and only if
\begin{enumerate}
\item the prime divisors of $d_B$ do not split $K$,
\item the prime divisors of $N$ split $K$, and
\item $c$ is prime to $d_B$.
\end{enumerate}
Note that these conditions determine $d_B$ in terms
of $E$ and $K$.

The modularity theorem,
the
Jacquet--Langlands correspondence
and Faltings's isogeny theorem
imply
that there is a surjective $\mathbold{Q}$-map
$\varphi : \Jac X_B \rightarrow E$
arising, as in \sec\ref{sec:relat-probl-comp},
from some newform $f_B \in \mathcal{F}_2(\Gamma ')$.
Such maps are not unique,
since post-composing one with a $\mathbold{Q}$-isogeny
$E \rightarrow E$
gives another,
but we may and shall pin down the $\Aut(E)$-orbit
of $\varphi$ 
by requiring that it not factor through a
$\mathbold{Q}$-isogeny $E \rightarrow E$ of degree $> 1$.

To a CM point $\zeta \in \CM_D(X_B)$
and an Atkin--Lehner operator $\tau \in \AL(\Gamma ')$,
we attach a point
$P(D,\tau) \in E(H_D)$
obtained as the image
of $[\tau \zeta] - [\zeta] \in \Pic^0(X_B)(H_D) \cong \Jac(X_B)(H_D)$
under
$\varphi$.
Although $P(D,\tau)$ depends upon $\zeta$,
its
orbit under $\Aut(E) \times
\Gal(H_D/K)$
depends only upon $D$ and $\tau$.
To classify such orbits conveniently, we assume that
$E$ is given by an affine equation $y^2 + a_1 x y + a_3 y = x^3
+ a_2 x^2 + a_4 x + a_6$ with coefficients in $\mathbold{Q}$
(so that the $\{\pm 1\}$-orbits on $E$ are classified
by $x$-coordinates),
and also that $E$ admits complex multiplication by neither
$\mathbold{Q}(i)$ nor $\mathbold{Q}(\sqrt{-3})$
(so that $\Aut(E) = \{\pm 1\}$).\footnote{The latter
condition is imposed strictly for convenience
(see \cite[\S6.1]{MR0314766}).}
The Galois orbit of an algebraic number is
classified by its minimal polynomial.
Letting $x(P(D,\tau))
\in H_D$ denote the $x$-coordinate of $P(D,\tau)$, the
following computational problem is then meaningful:
\begin{problem}\label{prob:compute-stuff}
  For $(E, D, \tau)$ as above,
  find the minimal polynomial
  of $x(P(D,\tau))$.
\end{problem}

The novelty here is that we are able
to address this problem via archimedean analysis even when
$d_B \neq 1$
(see \sec\ref{sec:motiv-numer-comp}).

\subsection{Method}\label{sec:numerical-examples-method}
Our method for addressing Problem \ref{prob:compute-stuff}
is the natural one suggested by
the content of \sec\ref{sec:relat-probl-comp}
and Theorem \ref{thm:rough-main-2}.
Let the newform
$f_B \in \mathcal{F}_k(\Gamma ')$
realize
$\varphi : \Jac X_B \rightarrow E$,
as in \sec\ref{sec:relat-probl-comp}.
Suppose that $f_B$
is compatibly normalized with its Jacquet--Langlands lift $f$ 
as in Definition \ref{defn:compatibly-normalized},
so that Theorem \ref{thm:rough-main-2} applies.
The Fourier coefficients $a_f(n)$ may be read off
from the arithmetic of $E$ over finite fields
in the usual way.
\begin{enumerate}
\item We find a basis for $H_1(X_B,\mathbold{Z})$,
  which amounts to finding generators for (the abelianization of)
  $\Gamma'$.  In the examples below, we either found such
  generators in the paper \cite{MR2019012} or worked them out
  ``by  hand''.  In more complicated examples, we
  could use the algorithm \cite{MR2541438} as implemented in
  MAGMA.
\item
  Using Theorem \ref{thm:rough-main-2},
  we compute to high precision
  some generators\footnote{In practice,
    it suffices to compute only one generator to high
    precision.}
  for the lattice $\Lambda_{f_B} :=
  \left\{ \int _\gamma f_B(z) \, d z : \gamma \in
    H_1(X_B,\mathbold{Z}) \right\}
  =
  \left\{ \int _{z_0}^{\gamma z_0}
    f_B(z) \, d z :
    \gamma \in \Gamma '
  \right\}$
  ($z_0 \in \mathbold{H}$ fixed).

\item Using the built-in functionality
  of SAGE/PARI,
  we choose a
  Weierstrass parametrization $\mathbold{C}/\Lambda_E \cong
  E(\mathbold{C})$ 
  for $E$.
\item We compute a minimal isogeny $\mathbold{C} / \Lambda_{f_B}
  \rightarrow \mathbold{C} / \Lambda_E$,
  or equivalently, a nonzero complex number $\mu$ for which
  $\mu \Lambda_{f_B}$ is contained in $\Lambda_E$ with minimal
  index.
  Our assumption $\Aut(E) = \{\pm 1\}$
  implies that $\mu$ is determined up to
  $\pm 1$.
\item
  For $\tau \in \AL(\Gamma ')$
  and $\gamma_1, \gamma_2 \in \Gamma '$,
  we compute
  $\int _{\gamma_1 \zeta } ^{\gamma_2 \tau \zeta }
  f_B(z) \, d z$
  via Theorem \ref{thm:rough-main-2}
  and its
   image
  $P(\zeta,\tau)$
  under
  $\mathbold{C} \rightarrow \mathbold{C} / \Lambda_{f_B}
  \xrightarrow{\mu}
  \mathbold{C}/\Lambda_E \cong E(\mathbold{C})$
  via the \emph{elliptic\_exponential} functionality of SAGE/PARI.
  Note that $P(\zeta,\tau)$ is independent of $\gamma_1$
  and $\gamma_2$,
  which we choose to minimize the length
  of the integration contour.
\item
  Let
  $\zeta_1, \dotsc, \zeta_{h(D)}$ be the
  $\Gal(H_D/K)$-orbit of some
  $\zeta \in \CM_D(X_B)$.
  We compute these
  as the fixed points for the $\Pic(\mathcal{O})$-orbit
  of some optimal embedding
  $\mathcal{O} \hookrightarrow R$.
  The minimal polynomial
  of $x(P(D,\tau))$
  divides
  $\prod (t - x(P(\zeta_i,\tau)))
  \in \mathbold{Q}[t]$,
  whose coefficients we compute precisely.
\end{enumerate}

\subsection{Results}\label{sec:numerical-examples-results}
We consider the following
(Cremona-labeled) elliptic
curves:
\begin{center}
  \begin{tabular}{|l|l|}
     \hline
    $E$ & model
    \\ \hline \hline
    $15A_1$
    &
    $y^2+x y + y = x^3 + x^2 - 10 x - 10$
    \\ \hline
    $21A_2$
    &
    $y^2 + x y = x^3 - 49 x - 136$
    \\ \hline
    $35 A_1$
    &
    $y^2 + y = x^3 + x^2 + 9 x + 1$
    \\ \hline
    $26 A_1$
    &
    $y^2 + x y + y = x^3 -5x -8$
    \\ \hline
    $26 B_2$
    &
    $y^2 + x y + y = x^3 - x^2 - 213 x -1257$
    \\ \hline
  \end{tabular}
\end{center}
The table below records some Heegner points that we found on
these curves, all coming from uniformizations by compact Shimura
curves ($d_B \neq 1$).
The points not labeled as torsion have infinite order.
\begin{center}
\begin{tabular}{|l|l|l|l|l|}
 \hline
\quad $E$
&
\quad $D$
&
$h$
&
\, $\tau$
&
\quad minimal polynomial of  $x(P(D,\tau))$
\\ \hline \hline
$15 A_1$
&
$-7$          &   $1$
&
$\tau_3$
&
$x = -4$
\\ \hline
&
& &
$\tau_5$
&
$x = 3$
\text{ ($2$-torsion)}
\\ \hline
&
& &
$\tau_{15}$
&
$x = 4/7$
\\ \hline
&
$-7 \cdot 2^2$
&   $1$
&
$\tau_3$
&
$x = -1108/25 = -5^{-2}  2^2 277^1$
\\ \hline
&
$-40$          &   $2$
&
$\tau_3$
&
$x = -29/32 $ 
\\ \hline
&

& &
$\tau_5$
&
$x = 3$ ($2$-torsion)
\\ \hline
&
$-43$          &   $1$
&
$\tau_3$
&
$x = -1339/256 = - 2^{-8}  13^1 103^1$
\\ \hline
&
& &
$\tau_5$
&
$x = 3$
\text{ ($2$-torsion)}
\\ \hline
&
& &
$\tau_{15}$
&
$x =-79/2107 = - 7^{-2}  43^{-1} 79^1$
\\ \hline
&
$-52$          &   $2$
&
$\tau_3$
&
$2^{10}  5^2 x^2 - 67072 x + 202177$
\\ \hline
&
$-55$          &   $4$
&
$\tau_3$
&
$11x^2 + 2509x - 751$
\\ \hline
&
$-67$          &   $1$
&
$\tau_3$
&
$x = -461899/30976 = -2^{-8} 11^{-2}  127^1  3637^1$
\\ \hline
$21A_2$
&
$-15$          &   $2$
&
$\tau_7$
&
$x = -23/5$
\\ \hline
&
$-39$          &   $4$
&
$\tau_7$
&
$ 13x^2 - 103x + 433$
\\ \hline
&
$-43$          &   $1$
&
$\tau_7$
&
$ x =  -1588/225 = - 3^{-2} 5^{-2} 2^2 397^1$
\\ \hline
&
$-51$          &   $2$
&
$\tau_7$
&
$ x= -116/17 = -  17^{-1} 2^2 29^1$
\\ \hline
&
$-15 \cdot 2^2$
&   $2$
&
$\tau_7$
&
$x = -527/5 = - 5^{-1} 17^1 31^1$
\\ \hline
&
$-67$          &   $1$
&
$\tau_7$
&
$ x= -17972/4225 =  - 5^{-2} 13^{-2} 2^2  4493$
\\ \hline
$35 A_1$
&
$-7$          &   $1$
&
$\tau_5$
&
$x = - 1/7$
\\ \hline
&
$-8$          &   $1$
&
$\tau_5$
&
$x = - 49/8$
\\ \hline
&
$-23$          &   $3$
&
$\tau_5$
&
$23 x ^3 + 5 0 6 2 x ^2 + 3 9 5 1 x + 4 1 2 9 1$
\\ \hline
&
$-7 \cdot 2^2$
&   $1$
&
$\tau_5$
&
$x= -7401/7 = - 7^{-1}  3^1  2467^1$
\\ \hline
&
$-8 \cdot 2^2$
&   $2$
&
$\tau_5$
&
$x = - 49/8$
\\ \hline
&
$-4 3$          &   $1$
&
$\tau_5$
&
$x =-489/688= -{2 ^{-4}  43^{-1}}  { 3^1 163^1}$
\\ \hline
&
$-7 \cdot 3^2$
&   $4$
&
$\tau_5$
&
$7 x^2 - 13 x + 73$
\\ \hline
&
$-67$          &   $1$
&
$\tau_5$
&
$x = -202489/4288 = -{2 ^{-6}  67^{-1}}  { 7^1 28927^1}$
\\ \hline
$26 A_1$
&
$-59 $          &   $3$
&
$\tau_2$
&
$11^4 59^1 x^3 + 2708827 x^2 + 2483001 x + 1553249x$
\\ \hline
&
$-67$ &
$1$ &
$\tau_2$ &
$-2020489/2815675 = - 5^{-2} 41^{-2} 67^{-1} 2020489^1$
\\ \hline
$26 B_2$
& $-24$
& $1$
& $\tau_{13}$
& $-359/12 = -2^{-3} 3^{-1} 359^1$
\\ \hline
&
$-52$
& $2$
& $\tau_{13}$
& point at infinity
\\ \hline
\end{tabular}
\end{center}

We emphasize that although we have listed only the
(minimal polynomial of the) $x$-coordinate,
the $y$-coordinate may be recovered by solving a
quadratic equation.

\begin{remark}
  These points were evaluated numerically
  using a mixture of SAGE/PARI
  and C++/GSL,
  and then recognized via continued
  fractions with the naked eye.
  We do not claim that our results
  are correct,
  but it is likely that they are,
  because in all cases we obtained a point with the expected
  field of definition.
  In principle, one could carry out such computations provably correctly
  by bounding the error terms in our explicit formulas
  and the heights of the points $P(D,\tau)$
  (via Gross--Zagier/Zhang formulas, see
  \cite[\S5]{MR2473878}).
\end{remark}

\begin{example}
  For $(E,D) = (35 A_1,-23)$
  (resp. $(26 A_1, -59)$),
  we have $h(D) = 3$,
  so there are three Atkin--Lehner orbits
  of CM points of conductor $D$
  on the Shimura curve $X_0^{d_B}(1)$
  with $d_B = 35$ (resp. $26$).
  Letting $x_1, x_2, x_3$ denote the
  $x$-coordinates of the resulting points on $E(\mathbold{C})$,
  we compute numerically that
  $23^1 \prod (t-x_i)$ (resp. $11^4 59^1 \prod (t-x_i)$)
  is approximately
  \[
  23  t^3 + 5061.9999982\ldots  t^2 + 3950.9999999\ldots  t + 41290.9999609 \ldots
  \]
  (resp.
  \[
  11^4 59^1 t^3 + 2708826.9996 \dotsc t^2 + 2483000.9994 \dotsc
  t + 1553248.99991 \dotsc.)
  \]
  In each case we obtain approximately the integral-coefficient polynomial
  given
  above,
  whose roots
  generate a cubic subfield of the Hilbert class field $H_{D}$ of $\mathbold{Q} (\sqrt{D})$.
\end{example}
\begin{example}
  We have listed only one example
  in which our computation returned a torsion point for ``non-obvious''
  reasons: $E = 26 B_2$, $D = -52$
  and $\tau = \tau_{13}$,
  for which the rank of
  the quadratic twist $E_D$ is $0$.
\end{example}
\begin{example}
  For $E = 35 A$, $D = -63 = -7 \cdot 3^2$,
  we got a point defined over the ring class field
  $H_{-63} = \mathbold{Q} (\sqrt{- 7}, \sqrt{- 3})$
  of the order $\mathbold{Z} [3 (1 +  \sqrt{-7})/2] \subset \mathbold{Q}(\sqrt{- 7})$ of conductor $3$.
\end{example}


\section{General method}\label{sec:method-proofs}
Throughout this section, we assume the notation and definitions
of \sec\ref{sec:background}.

\subsection{Overview}\label{sec:overview-1}
Let $B$ be an indefinite rational quaternion algebra,
let $N$ be a positive integer coprime to $d_B$,
let $\Gamma ' = \Gamma_0^{d_B}(N)$,
let $R \subset B \hookrightarrow M_2(\mathbold{R})$
be the Eichler order and real embedding for which
$\Gamma ' = R \cap \SL_2(\mathbold{R})$
arises as (the image of) the group of norm one units in $R$,
and let
$\Gamma = \Gamma_0(d_B N)$.
Let $k$ be a positive even integer,
and let $f_B \in \mathcal{F}_k(\Gamma ')$
and $f \in \mathcal{F}_k(\Gamma)$ be compatibly normalized
Jacquet--Langlands
correspondents.

One can compute the values of $f_B$ 
in the following way:
\begin{enumerate}
\item[(I)] Write down an explicit form of the Shimizu correspondence.
  This expresses $f_B$ in the form
  \begin{equation}\label{eq:27}
    (y_1 y_2)^{k/2}
    \overline{f_B(z_1)} f_B(z_2)
    =
    \int _{\Gamma \backslash \mathbold{H} }
    y^k \overline{f(z)} \theta_{z_1,z_2}(z)
    \, \frac{d x \, d y}{y^2},
  \end{equation}
  where $\theta_{z_1,z_2}$ is an explicit non-holomorphic
  theta series
  depending only upon $R \subset B \hookrightarrow
  M_2(\mathbold{R})$
  and $k$.
  Thanks to
  a result
  of Watson \cite[\S 2.1]{watson-2008},
  doing so amounts to an exercise in ``unadelization.'' 
  We define $\theta_{z_1,z_2}$ in \sec\ref{sec:expl-shim-corr-1}
  and show that it satisfies \eqref{eq:27} in
  Appendix \ref{sec:shim-corr-1}.
\item[(II)]
  Compute the Petersson inner product
  on the RHS of \eqref{eq:27}.
  We discuss several approaches
  to this problem in \sec\ref{sec:some-cand-appr}.
  Our preferred approach (Theorem \ref{thm:compute-integral-of-automorphic-form}),
  which leads to the explicit formulas of \sec\ref{sec:explicit-formula},
  requires also the following
  intermediate step:
  \begin{enumerate}
  \item[(I$\tfrac{1}{2}$)]
    Write down the constant term
    in the Fourier expansion of $y^k \overline{f(z)}
    \theta_{z_1,z_2}(z)$
    at each cusp of $\Gamma$.
    We carry this out 
    in Appendix \ref{sec:fourier-expansions}
    when $N$ is squarefree.
  \end{enumerate}
\end{enumerate}

After stating a precise form of the identity \eqref{eq:27} in
\sec\ref{sec:expl-shim-corr-1}, we turn to a discussion
of the general
problem of computing integrals of automorphic
functions (e.g., Petersson inner products).  This discussion
occupies the remainder of \sec\ref{sec:method-proofs} and
serves to motivate Theorem
\ref{thm:compute-integral-of-automorphic-form}, which
gives a formula for rather general such integrals
via roughly a ``hybrid Rankin-Selberg and
approximate functional equation'' approach.
In Appendix \ref{sec:comp-shim-lift},
we apply the
procedure outlined above to prove the formulas stated in
\sec\ref{sec:explicit-formula}.
\begin{remark}
  This procedure applies also to more general automorphic
  quotients
  (such as Shimura curves associated
  to quaternion algebras over totally real fields),
  but we reserve further comments in that direction for a future paper.
\end{remark}

\subsection{Definition of Shimizu theta function}\label{sec:expl-shim-corr-1}
Let notation be as in \sec\ref{sec:overview-1}.
Recall the functions $X$ and $P$
defined in \sec\ref{sec:weighted-sums-over}.
For each lattice $L < M_2(\mathbold{R})$
and positive reals $t$ and $y$,
set
\[
U_k(L;t,y)
:=
\frac{1}{2}
y
\sum_{\substack{
    \alpha \in L \\
    \det(\alpha) = t
  }
}
X(\alpha)^k
e^{- 2 \pi y P(\alpha)}.
\]
Fix $z_1, z_2 \in \mathbold{H}$,
recall the notation
$\sigma_{z_i}$ from \sec\ref{sec:main-formula},
and define\footnote{See Appendix \ref{sec:fourier-expansions}
for the Fourier expansion of $\theta_{z_1,z_2}$
at other cusps.}
\[
\theta_{z_1,z_2}(z)
:=
\sum_{n \in \mathbold{Z}}
U_k(\sigma_{z_1}^{-1} R \sigma_{z_2};n,y)
e(n x).
\]
\begin{theorem}[\cite{watson-2008}, Appendix \ref{sec:shim-corr-1}]
  \label{thm:explicit-shimizu-unadelized}
  We have $\theta_{z_1,z_2}|_k \gamma = \theta_{z_1,z_2}$
  for all $\gamma \in \Gamma$,
  and the identity \eqref{eq:27} holds.
\end{theorem}

\begin{remark}
  \label{rmk:nonsplit-shimizu-theta-is-cuspidal}
  The quadratic form $(B,\det)$
  represents $0$ (nontrivially) if and only if $B$ is split,
  so that
  $U_k(L;0,y) \neq 0$
  for some lattice $L < B$
  if
  only if $B \cong M_2(\mathbold{Q})$.
\end{remark}

\subsection{Problem: computing integrals of automorphic functions}
\label{sec:comp-integr-autom}
Let $M$ be a positive integer and $\Gamma =
\Gamma_0(M)$.\footnote{
  The discussion that follows
  applies reasonably well
  to any lattice $\Gamma$ with cusps.}
There arose in \sec\ref{sec:overview-1}
the problem of computing the Petersson inner product
of a fixed newform against an explicit theta series
on $\Gamma \backslash \mathbold{H}$.
This problem is a special case of a more general one
to which we now turn.

Suppose that we are given
a $\Gamma$-invariant function
$F : \Gamma \backslash \mathbold{H} \rightarrow \mathbold{C}$.
By ``given,'' we mean that we know
the Fourier expansion
\begin{equation}\label{eq:22}
  F(\tau z)
  = \sum_{n \in \mathbold{Q}}
  a_F(n,y;\tau)
  e(n x)
\end{equation}
of $F$, for each $\tau \in \SL_2(\mathbold{Z})$,
to some large but finite precision.
The main example to keep in mind
is
when
$F(z) = y^k \overline{f_1(z)} f_2(z)$
for suitably regular automorphic functions $f_1$ and $f_2$
of the same weight $k$,
at least one of which is neither holomorphic
nor a Hecke eigenform.
Other approaches
are available
when, say,  both $f_1$ and $f_2$ are holomorphic.

We do not assume any regularity of $F$, but we do assume
that it is absolutely integrable on $\Gamma \backslash
\mathbold{H}$.
We assume also
that one can compute
$a_F(n,y;\tau)$, to a given precision,
in essentially bounded time for large $y$
and in time $(1/y)^{O(1)}$ for small $y$.

Given such an $F$, how can we compute
the integral
\[
\int F := \int_{\Gamma \backslash \mathbold{H}} F(z) \, \frac{d x \, d y}{
  y^2}?
\]
In applications
we might wish to compute many
such integrals (say a million).

\subsection{Some candidate approaches}\label{sec:some-cand-appr}
As motivation,
we describe here
some natural
approaches to the problem stated in \sec\ref{sec:comp-integr-autom}.

\subsubsection*{Candidate approach 1: brute force}
Suppose first that $M = 1$,
and that the Fourier expansion \eqref{eq:22}
converges rapidly on the standard
fundamental domain
\begin{equation}\label{eq:10}
  \mathcal{F}
  :=
  \left\{ z = x + i y \in \mathbold{H}
  : |x| \leq 1/2, |z| \geq 1
\right\}
\end{equation}
for $\SL_2(\mathbold{Z})$.
The latter condition
holds in all applications we have in mind,
so we henceforth refrain from stating it explicitly
in this informal discussion.
Then we can compute $\int F$
by truncating its Fourier expansion
and integrating each term over $\mathcal{F}$.

For general $M$,
we can take as a fundamental domain
for $\Gamma \backslash \mathbold{H}$
the essentially disjoint union of $\tau \mathcal{F}$
over $\tau \in \SL_2(\mathbold{Z}) / \Gamma$.\footnote{In implementations,
  the sum over $\tau \in \SL_2(\mathbold{Z}) / \Gamma$
  is replaced by a sum over
  the cusps of $\Gamma$.}
Then $\int F$ is the sum
over $\tau$ of the integral
of $F(\tau z)$ over $\mathcal{F}$,
which we compute as above.
\subsubsection*{Candidate approach 2: holomorphic projection}
Suppose that $k > 2$, and that $F$ is a product
$F(z) = y^k \overline{f_1(z)} f_2(z)$
where $f_1 \in S_k(\Gamma)$ is a holomorphic
cusp form.  Suppose also
there exists $\delta > 0$
such that that $f_2(\tau z) \ll y^{-\delta}$
for all $\tau \in \SL_2(\mathbold{Z})$, $y \geq 1$.\footnote{This assumption holds in the setting of \sec\ref{sec:overview-1}
provided
that $B \not \cong M_2(\mathbold{Q})$.}
Let $\langle , \rangle$ denote the Petersson ``inner product''
on weight $k$ automorphic functions on $\Gamma$,
given by
$\langle h_1,h_2 \rangle
= \int _{\Gamma \backslash \mathbold{H} }
y^k \overline{h_1(z)} h_2(z) \, \frac{d x \, d y}{y^2}$
whenever the integral converges absolutely.
The pairing $\langle , \rangle$ induces the structure of a finite-dimensional
Hilbert space on $S_k(\Gamma)$,
so there exists a unique form $h_2 \in S_k(\Gamma)$,
called the \emph{holomorphic projection}
of $f_2$, with the property that
$\langle h_1, h_2 \rangle = \langle h_1, f_2 \rangle$
for all $h_1 \in S_k(\Gamma)$.
Suppose that the Fourier expansion
of $f_2$ at $\infty$ reads
$f_2(z)
= \sum a_{f_2}(n,y) e(n x)$,
with the sum taken over $n \in \mathbold{N}$.
Then,
following Gross--Zagier \cite[\S IV.5]{GZ86},
its holomorphic projection $h_2$
has the expansion
$h_2(z) = \sum a_{h_2}(n) e(n z)$,
where
\begin{equation}\label{eq:23}
  a_{h_2}(n)
  =
  \frac{
    \int _{\mathbold{R}_+^\times }
    y^{k/2} a_{f_2}(n,y) \cdot y^{k/2} e^{- 2 \pi n y} \, \frac{d^\times y}{y}
  }{
    \int _{\mathbold{R}_+^\times }
    y^{k/2} e^{- 2 \pi n y} \cdot y^{k/2} e^{- 2 \pi n y} \, \frac{d^\times y}{y}
  }.
\end{equation}
By linear algebra on the finite-dimensional
space $S_k(\Gamma)$, we may write
$h_2 = c f_1 + h_2'$
for some complex number $c$
and some element $h_2$ of $S_k(\Gamma)$ that is
$\langle , \rangle$-orthogonal
to $f_1$.
The quantity $\int F = \langle f_1, f_2 \rangle
= \langle f_1, h_2 \rangle = c \langle f_1, f_1 \rangle$
may now be computed
in a number of ways
(see e.g. Example \ref{example:computing-petersson-norms}).

\subsubsection*{Candidate approach 3: equidistribution of
  thickened horocycles}
Our eventual
method of choice
(Theorem \ref{thm:compute-integral-of-automorphic-form} below)
is natural from several perspectives
in retrospect,
but was originally inspired by the following
technique used by Holowinsky \cite{MR2680498}
in his work on the quantum unique ergodicity conjecture.
Suppose, for simplicity, that $M = 1$.
Let $Y$ be a real parameter
tending to $\infty$.
The thickened horocycle
\[
\mathcal{R}_Y
:= \{x + i y \in \mathbold{H} : |x| \leq 1/2, y \in [Y^{-1}, 2
Y^{-1}]\}
\]
contains $\asymp Y$ copies of the fundamental domain
$\mathcal{F}$
for $\SL_2(\mathbold{Z}) \backslash \mathbold{H}$
(see \eqref{eq:10}),
so it is reasonable to expect that
\[
\int F
\approx 
\frac{1}{Y}
\int_{z \in \mathcal{R}_Y} F(z) \, \frac{d x \, d y}{y^2}
=
\frac{1}{Y}
\int _{y = Y^{-1}}^{2 Y^{-1}}
a_F(0,y;1) \, \frac{d^\times y}{y}.
\]
Intuitively, $\int F$
should be roughly an average of the constant term $a_F(0,y;1)$
taken over $y$ sufficiently close to $0$.
This is made rigorous in the following lemma
implicit in \cite{MR2680498}
(although the formulation here is our own).
\begin{lemma}
  \label{lem:holow-regularize}
  Fix a smooth compactly-supported
  function $h \in C_c^\infty(\mathbold{R}^\times_+)$
  with Mellin transform
  $h^\wedge(s)= \int _{(2)} h(y) y^{-s} \, \frac{d s}{2 \pi i}$
  normalized so that
  $h^\wedge(1) = \int_{\SL_2(\mathbold{Z}) \backslash \mathbold{H}}
  \, \frac{d x \, d y}{y^2}$.
  Define a norm $S$ on the space of functions $F :
  \SL_2(\mathbold{Z}) \backslash \mathbold{H} \rightarrow \mathbold{C}$
  by
  $S(F)
  = \int _{z \in \mathcal{F}}
  y^{1/2} |F(z)| \, \frac{d x \, d y}{y^2}.$
  Let $F : \SL_2(\mathbold{Z}) \backslash \mathbold{H} \rightarrow \mathbold{C}$
  be a function for which
  $S(F) < \infty$,
  and let $a_F(0,y;1)$ be its zeroth Fourier coefficient
  at $\infty$ as in \eqref{eq:22}.
  Then for each $Y \geq 1$, we have
  \begin{equation}\label{eq:24}
    \int_{\SL_2(\mathbold{Z}) \backslash \mathbold{H}} F(z) \, \frac{d x \, d y}{y^2}
    =
    \frac{1}{Y}
    \int_{y \in \mathbold{R}_+^*}
    h(Y y)
    a_F(0,y;1)
    \, \frac{d^\times y}{y}
    + O \left(
      \frac{S(F)}{Y^{1/2}}
    \right),
  \end{equation}
  where the implied constant depends only upon $h$.
\end{lemma}

\subsection{Criticism of candidate approaches}\label{sec:crit-cand-appr}
The approaches described in \sec\ref{sec:some-cand-appr}
work,\footnote{We have experimented with
the first two (using SAGE) for $F$ arising as a rather general
product of squarefree level
newforms, their derivatives, and Shimizu theta series;
see \cite{nelson-appendix-C-prasanna-notes}.}
but suffer some drawbacks,
which we now describe.

The brute force approach, which involves integrating
over $\mathcal{F}$, is somewhat slow.
In applications, we are not aware of closed form
expressions
for the integrals
over $\mathcal{F}$ 
that arise, so we must compute each one afresh.
A typical such integral might look like a linear combination
of terms
\begin{equation*}
  \mathop{\int _{y = 0 } ^\infty  \int _{x = -1/2} ^{1/2} }
  _{x ^2 + y ^2 \geq 1}
  y^{k-1} e^{2 \pi i n b} e^{- 2 \pi c y}
  \, d x \, d y
\end{equation*}
over some integers $b$ and nonnegative reals $c$.  
Caching helps, but not enough
to carry out efficiently certain more advanced applications,
such as Heegner point
computations.

Moreover, we have in mind the natural generalization
of the problem under consideration
to the setting of automorphic forms on $\GL_2$ over a number
field,
in which the integral over $\Gamma_0(M) \backslash
\mathbold{H}$
with its explicit fundamental domain
becomes replaced by an integral over
a $(2 r_1 + 3 r_2)$-dimensional quotient
$\Delta \backslash (\mathbold{H}^{r_1} \times
\mathbold{H}_3^{r_2})$, or perhaps several copies
of such quotients;
here $\mathbold{H}_3$
is hyperbolic upper half-space.
An explicit description of a fundamental domain
for $\Delta \backslash (\mathbold{H}^{r_1} \times
\mathbold{H}_3^{r_2})$
seems prohibitively complicated,
even for the simplest nontrivial number fields.
Even granting such a description, the problem of integrating an
oscillating
multivariate function over such a fundamental domain
(a few million times)
still seems computationally infeasible.

The holomorphic projection approach
works
quite well when $k$ is  large enough,
even over number fields.
It is simpler than the other methods
in the presence of nontrivial level
because it requires knowledge of the Fourier
coefficients $a_F(n,y;\tau)$ only when $\tau = 1$
(i.e., at the cusp $\infty$).
When $k$ is
small,
say $k = 6$ or $k = 4$, the
computation of the numerator of \eqref{eq:23}
becomes prohibitively slow.
There is a holomorphic projection
formula for $k = 2$ that we have not stated here
(see \cite[\S IV.6]{GZ86}),
but it involves a delicate limiting procedure
that seems to render it computationally ineffective.
For us, the case $k = 2$ is important
owing to its relevance for computing
modular parametrizations of elliptic curves.

Lemma \ref{lem:holow-regularize}
is convenient for certain theoretical applications,
as in \cite{MR2680498}.
However, it is not well-suited for numerics:
One must take
$Y^{1/2} \gg 10^R$
in Lemma \ref{lem:holow-regularize}
to obtain 
$R$ digits of precision past the decimal point,
but it may require time (not less than) polynomial in $Y$
to compute the integral on the RHS of \eqref{eq:24}.
For applications in which high precision is needed,
an algorithm requiring time exponential in the number
of digits of desired precision is infeasible.


\subsection{Our method of choice}\label{sec:our-method-choice}
We begin by recalling some background on Eisenstein series.
Let $\xi(s) = \Gamma_\mathbold{R}(s) \zeta(s)$
be the completed Riemann zeta function,
where $\Gamma_\mathbold{R}(s) = \pi^{-s/2} \Gamma(s/2)$
and $\zeta(s) = \sum_{n \in \mathbold{N}} n^{-s}$ ($\Re(s) >
1$).
Let
$\Gamma_\infty
= \left\{ \pm
  \left(
    \begin{smallmatrix}
      1&n\\
      &1
    \end{smallmatrix}
  \right)
  : n \in \mathbold{Z}  \right\}$
be the stabilizer
in
$\SL_2(\mathbold{Z})$
of the cusp $\infty$,
let
\[
E_s(z) = \sum_{\gamma \in \Gamma_\infty \backslash
  \SL_2(\mathbold{Z})}
\Im(\gamma z)^s
\quad (\Re(s) > 1)
\]
be the standard
real-analytic Eisenstein series
for $\SL_2(\mathbold{Z})$,
and let
$E_s^*(z) = 2 \xi (2 s) E_s(z)$
be its completion.
The functions $E_s$ and $E_s^*$ descend
to $\SL_2(\mathbold{Z}) \backslash \mathbold{H}$.
We collect some of their standard properties
in the following proposition.
\begin{proposition}[see \cite{MR1942691}]
\label{prop:eis-standard-facts}
  Let $z \in \mathbold{H}$.
  The function $s \mapsto E_s(z)$, defined initially
  for $\Re(s) > 1$
  by an absolutely and locally uniformly convergent series, extends to a meromorphic
  function on the complex plane that is holomorphic
  in the half-plane $\Re(s) \geq 1/2$ away from a simple pole
  at $s=1$.
  The function $s \mapsto E_s^*(z)$
  extends to a meromorphic function on the complex plane,
  holomorphic away from simple poles at
  $s = 1$ and $s = 0$,
  that satisfies the functional equation
  $E^*_s(z) = E^*_{1-s}(z)$.
  One has $\res_{s=1} E_s(z) =
  \left(\int_{\SL_2(\mathbold{Z}) \backslash \mathbold{H}}
  \, \frac{d x \, d y}{y^2} \right)^{-1}$
  and $\res_{s=1} E_s^*(z) = 1$.
  The function $E_s(z)$ has moderate growth
  in both variables in the sense that
  for $\Re(s)$ in a fixed compact subset of $[1/2,\infty)$
  and $z = x + i y$ in the fundamental domain $\mathcal{F}$,
  it satisfies $E_s(z) \ll (1 + |s|)^{O(1)} y^{\Re(s)}$.
  The function $s \mapsto E_s^*(z)$ decays rapidly
  in vertical strips in the sense that
  if $A \in \mathbold{R}$,
  $z \in \mathbold{H}$
  and $\Re(s) \in \mathbold{R}$ are
  taken to lie in fixed compact sets,
  then
  $E_s^*(z) \ll (1 + |s|)^{-A} |s(s-1)|^{-1}$.
\end{proposition}
As motivation,
we sketch the proof of
Lemma \ref{lem:holow-regularize}.
Let $h$ and $Y$ be as in its statement,
and define
$E^Y(z) = \sum_{\gamma \in \Gamma_\infty \backslash
  \SL_2(\mathbold{Z})}
Y^{-1} h(Y \cdot \Im(\gamma z))$.
By unfolding,  $\int E^Y F$
gives the main term on the RHS of \eqref{eq:24}.
On the other hand,
Mellin inversion and a contour shift
show
that
\[
E^Y(z) = \int _{(2)} h^\wedge(s) Y^{s-1} E_s(z) \, \frac{d s}{2
  \pi i}
= 1 +
Y^{-1/2} \int _{t \in \mathbold{R}}
h^\wedge(1/2+it)
Y^{it}
E_{1/2+it}(z)
\, \frac{d t}{2 \pi }
\]
Since
repeated partial integration implies
that $h^\wedge(1/2+it) \ll (1+|t|)^{-A}$
for each fixed $A \in \mathbold{R}$,
we deduce from Proposition \ref{prop:eis-standard-facts}
that $E^Y(z) = 1 + O(Y^{-1/2} y^{1/2})$
for all $z \in \mathcal{F}$.
Integrating against $F$
gives \eqref{eq:24}.

We refine this argument by completing
the Eisenstein series and continuing past the line $\Re(s) =
1/2$,
noting that unlike
$E_s$,
which has infinitely many poles in the region $0 < \Re(s) < 1/2$,
the function $E_s^*$ is holomorphic
away from $s = 0$ and $s = 1$.
Fix $\eps > 0$,
and let $H$ be a holomorphic function
on the strip $\{s \in \mathbold{C} : -\eps  < \Re(s) < 1 + \eps 
\}$ satisfying the conditions
$H(0) = 0$, $H(1) = 1$, and
$H(s) \ll (1 + |s|)^{O(1)}$.
For example,
one can take $H(s) = s$.
Then $H(s) E_s^*$
is holomorphic away from its simple pole at $s = 1$
of residue $1$ and decays rapidly
in the strip $-\eps  < \Re(s) < 1 + \eps $,
so for each $\delta \in (0,\eps )$,
we have the fundamental identity
\begin{align}\label{eq:26}
  1 = \res_{s = 1}
  H(s) E_s^*
  &=
  \left(
    \int _{(1 + \delta )} - \int _{(- \delta )}
  \right)
  H(s) E_s^*
  \, \frac{d s}{2 \pi i} \\ \nonumber
  &= 
  \int _{(1 + \delta )}
  \left( H (s) - H (1 - s ) \right)
  E_s^* \, \frac{d s}{2 \pi i}.
\end{align}
In the final step, we used the functional
equation $E_s^* = E_{1-s}^*$.

We would like to integrate \eqref{eq:26}
against $F : \Gamma \backslash \mathbold{H} \rightarrow \mathbold{C}$
and interchange the order of integration
to obtain
\begin{align*}
  \int F
  &=
  \int _{(1 + \delta )} (H(s) - H(1-s)) 2 \xi(2 s)
  \left(
    \int F E_s
  \right)
  \, \frac{d s}{2 \pi i}.
\end{align*}
By taking $\delta$ sufficiently small,
we see that this interchange is justified
provided that $\lvert F(z) \rvert \ll y^{-\alpha}$
for some fixed $\alpha > 0$ and
all $z = x + i y$ with $y \geq 1$.
Unfolding gives
$\int F E_s = \int_{y \in \mathbold{R}^\times_+}
a_0(0,y;1) y^{s-1} \, d^\times y$.

If $\Gamma$ is a finite
index subgroup of $\SL_2(\mathbold{Z})$
and $F$ is on $\Gamma \backslash \mathbold{H}$,
then applying the above argument
to its pushforward
$z \mapsto \sum_{\Gamma \backslash \SL_2(\mathbold{Z})}
F(\tau z)$ on $\SL_2(\mathbold{Z}) \backslash \mathbold{H}$
yields:
\begin{theorem}\label{thm:compute-integral-of-automorphic-form}
  Let $\Gamma$ be
  a finite
  index subgroup of $\SL_2(\mathbold{Z})$.
  Let $F : \Gamma \backslash \mathbold{H} \rightarrow
  \mathbold{C}$
  be a bounded measurable function satisfying
  $F(\tau z) \ll y^{-\alpha}$ for some fixed $\alpha > 0$,
  almost all $z = x + i y$ with $y \geq 1$,
  and all $\tau \in \SL_2(\mathbold{Z})$.
  Let $\eps > 0$, and let $H$ be a holomorphic function
  on $\{s \in \mathbold{C} : -\eps < \Re(s) < 1+\eps\}$
  satisfying $H(0) = 0$, $H(1) = 1$, and $H(s) \ll (1+|s|)^{O(1)}$.
  Then for $\delta \in (0,\min(\alpha,\eps))$, we have
  \[
  \int _{\Gamma \backslash \mathbold{H} }F(z)
  \, \frac{d x \, d y}{y^2}
  = 
  \int _{(1+\delta)}
  (H(s) - H(1-s)) 2 \xi(2 s)
  \sum_{\tau \in \Gamma \backslash \SL_2(\mathbold{Z})}
  a_F(0,\cdot;\tau)^\wedge(1-s)
  \, \frac{d s}{2 \pi i}.
  \]
  Here
  $a_F(0,y;\tau)$ is the constant term of the Fourier expansion
  of $F(\tau z)$ as in \eqref{eq:22},
  while
  \[
  a_F(0,\cdot;\tau)^\wedge(1-s)
  := \int _{y \in \mathbold{R}^\times _+} a_F(0,y;\tau)
  y^{s-1} \, d^\times y.
  \]
\end{theorem}

\begin{example}\label{example:computing-petersson-norms}
  We give an example
  in which Theorem
  \ref{thm:compute-integral-of-automorphic-form} specializes
  to
  the approximate functional equation
  applied to Rankin--Selberg $L$-functions.
  Take $\Gamma = \SL_2(\mathbold{Z})$, $k = 12$, and let
  $\Delta(z) =  q \prod ( 1 - q^n)^{24}
  = \sum a_\Delta(n) q^n \in S_k(\Gamma)$ be as in \sec\ref{sec:summary-results}.
  Define $F : \Gamma \backslash \mathbold{H} \rightarrow
  \mathbold{C}$
  by the formula $F(z) = y^k \lvert \Delta(z)  \rvert ^2$.
  The constant term of $F$ is
  $a_F(0,y;1)
  = y^k \sum
  \lvert a_\Delta(n) \rvert^2 e ^{- 4 \pi n y}$,
  which has the Mellin transform (for $\Re(s) > 1$)
  \[
  a_F(0,\cdot;1)^\wedge(1-s)
  =
  \frac{\Gamma(k-1+s)}{(4 \pi)^{k-1+s}}
  \sum_{n \in \mathbold{N}} \frac{\left\lvert a_\Delta(n)
    \right\rvert^2}{n^{k-1+s}}
  =:
  \frac{\Gamma(k-1+s)}{(4 \pi )^{k-1+s}}
  \frac{L(\Delta \times \Delta,s)}{\zeta(2 s)}.
  \]
  Thus Theorem \ref{thm:compute-integral-of-automorphic-form}
  reads
  \begin{align*}
    \int_{\SL_2(\mathbold{Z}) \backslash \mathbold{H}}
    y^k \lvert \Delta (z) \rvert^2
    \frac{d x \, d y }{y^2}
    &=
    \int _{(1 + \delta )}
    (H(s)-H(1-s))
    2 \xi (2 s )
    a_F(0,\cdot;1)^\wedge(1-s)
    \, \frac{d s}{2 \pi i}
    \\
    &= 2
    \sum _{n \in \mathbold{N} }
    \frac{\left\lvert a _\Delta (n) \right\rvert ^2 }{(4 \pi n)
      ^{k - 1 }}
    \sum_{d \in \mathbold{N}}
    \Phi_{k,H}(4 \pi^2 d^2 n ),
  \end{align*}
  where
  $\Phi_{k,H}(y)
  = \int _{(1 + \delta )}
  (H (s) - H (1 - s )) \Gamma(s) \Gamma(k-1+s) y^{-s}
  \, \frac{d s}{2 \pi i}.$
  For the choice $H(s) = s$, one can verify that
  $\Phi_{k,H}(y)
  =
  y^{(k-1)/2} ( 4 \sqrt{y} K_{k-2}(2 \sqrt{y})
  - 2 K_{k-1}(2 \sqrt{y}))$
  (see Lemma \ref{lem:evaluate-integral-bessel-stuff}).
  The result obtained
  is equivalent to what one would get
  by first expressing the Petersson norm of $\Delta$ in terms
  of the residue of $L(\Delta \times \Delta, s)$
  at $s=1$ via the Rankin--Selberg method, and then computing that
  residue
  via an approximate functional
  equation.
  The ``two proofs'' are essentially
  rearrangements
  of one another.

\end{example}

\begin{remark}\label{rmk:holom-proj-wt-2}
  Theorem \ref{thm:compute-integral-of-automorphic-form} allows
  one to compute effectively the holomorphic projection of
  certain weight $k$ automorphic functions, even in the delicate
  case $k=2$.  However, it requires knowledge of Fourier
  expansions at all cusps.  It would be interesting to have a
  effective method for computing holomorphic projections using
  Fourier expansions only at the cusp $\infty$ when $k = 2$.
  Conceivably one could redo the above argument with
  the non-holomorphic weight $k$ Poincar\'{e} series and its
  recurrence relation under $s \mapsto s+1$ taking the place of the non-holomorphic
  weight $0$ Eisenstein series and its functional equation under $s \mapsto
  1-s$,
  but we have not attempted to do so.
\end{remark}

\begin{remark}
  Theorem \ref{thm:compute-integral-of-automorphic-form} is also
  of theoretical interest, because it gives a way to evaluate
  integrals of automorphic functions $F$ for which one cannot
  satisfactorily control the norm $S(F)$ of Lemma
  \ref{lem:holow-regularize}. We exploit this property
  in a paper under preparation to bound
  periods of restrictions of Hilbert modular forms.
\end{remark}

\appendix

\section{Proof of formulas}
\label{sec:proofs-formulas}
In this section we show
how the identity of Theorem
\ref{thm:explicit-shimizu-unadelized}
and the method for computing Petersson
inner products afforded by Theorem
\ref{thm:compute-integral-of-automorphic-form}
imply the identites stated in \sec\ref{sec:explicit-formula}.

\subsection{Fourier expansions at various
  cusps}\label{sec:fourier-expansions}
Preserve the notation of \secs\ref{sec:overview-1}
and \ref{sec:expl-shim-corr-1}.
Let $F : \mathbold{H} \rightarrow \mathbold{C}$
be the function
$F(z) = y^k \overline{f(z)} \theta_{z_1,z_2}(z)$,
which arose as the integrand of \eqref{eq:27}.
Owing to the identity
$\Im(\tau z) = \Im(z) / |c z + d|^2$
for $\tau = \left(
  \begin{smallmatrix}
    *&*\\
    c&d
  \end{smallmatrix}
\right) \in \SL_2(\mathbold{R})$,
we have
$F(\tau z)
=
y^k \overline{f|_k \tau(z)}
\theta_{z_1,z_2}|_k \tau(z)$
for all $\tau \in \SL_2(\mathbold{R})$.
In particular,
$F(\gamma  z) = F(z)$ for all $\gamma  \in \Gamma$.
For each $\tau \in \SL_2(\mathbold{Z})$,
write
$F(\tau z)
= \sum
a_F(n,y;\tau) e(n x)$
for the Fourier expansion of $F(\tau z)$,
where the sum is over $n$ in $(1/w) \mathbold{Z}$
with $w$ the width of the cusp $\tau \infty$.

Suppose now that $N$, the level of the Eichler order
$R$,
is squarefree.
Our aim in this subsection is to evaluate
the sum, taken over $\tau \in \Gamma \backslash
\SL_2(\mathbold{Z})$, of the ``constant terms''
$a_F(0,y;\tau)$ of $F$.
For a lattice $L < M_2(\mathbold{R})$
and positive reals $t$ and $y$,
set
\[
\widetilde{U}_k(L;t,y)
:= y^k
U_k(L;t,y)
e^{- 2 \pi t y}.
\]
\begin{theorem}
  \label{thm:zeroth-fourier-coefficient-of-integrand}
  With $c_f(d)$ as in \eqref{eq:28},
  we have
  \[
  \sum_{\tau \in \Gamma_0(d_B N) \backslash \SL_2(\mathbold{Z})}
  a_F(0,y;\tau)
  = 
  \sum_{n \in \mathbold{N}}
  a_f(n)
  \sum_{d|d_B N}
  c_f(d)
  \widetilde{U}_k(\sigma_{z_1}^{-1} R^{(d)} \sigma_{z_2};n/d,y).
  \]
\end{theorem}

The proof is a calculation that occupies
the remainder of this subsection.

\subsubsection*{Coset representatives}
\label{sec:coset-repr}
For each positive integer $M$,
the natural map $\SL_2(\mathbold{Z}) \rightarrow
\SL_2(\mathbold{Z}/M)$
is surjective.
By the Chinese remainder theorem,
it follows that
for each divisor $d$ of $d_B N$,
there exists an element $\tau_d \in \SL_2(\mathbold{Z})$
such that
\[
\tau_d \equiv \begin{bmatrix}
  0 & 1 \\
  -1 & 0
\end{bmatrix}
\mod d,
\quad 
\tau_d \equiv \begin{bmatrix}
  1 & 0 \\
  0 & 1
\end{bmatrix}
\mod \frac{d_B N}{d}.
\]
The difference between two such elements
$\tau_d$
belongs to the normal subgroup $\Gamma(d_B N)
= \{\gamma \in \SL_2(\mathbold{Z}):
\gamma \equiv 1 \mod d_B N \}$ of $\Gamma$.
The width of the cusp $\tau_d \infty$
is $d$.
For each real $x$
and nonzero real $y$,
set
\[
n(x) = \begin{bmatrix}
  1 & x \\
  0 & 1
\end{bmatrix}
\in \SL_2(\mathbold{R}),
\quad
a(y) = \begin{bmatrix}
  y^{1/2} & 0 \\
  0 & y^{-1/2}
\end{bmatrix}
\in \SL_2(\mathbold{R}).
\]
Recall that
$\Gamma = \Gamma_0(d_B N)$
with $d_B N$ squarefree.
The following lemma is well known.
\begin{lemma}
  \label{lem:coset-repr-1}
  Let $d$ traverse the positive divisors
  of $d_B N$, and let  $j$
  traverse
  a set of representatives
  for $\mathbold{Z} / d$.
  Then the matrices
  $\tau_d n(j)$
  traverse a set of
  representatives for
  $\Gamma  \backslash
  \SL_2(\mathbold{Z})$.
\end{lemma}

\subsubsection*{Fourier expansion of $f$}
Let $d$ be a divisor of $d_B N$.
The newform $f \in \mathcal{F}_k(d_B N)$
has squarefree level.
Let $a_f(n)$ be the $n$th Fourier coefficient of $f$, as in \eqref{eq:7}.
By a well-known result of Atkin--Lehner,
it transforms under $\tau_d a(d)$
via
$f|_k \tau_d a(d) (z)
=
\frac{\mu(d) d^{k/2-1}}{a_f(d)}
f(z).$
Since $f |_k a(d)^{-1}(z)
= d^{-k/2} f(z/d)$,
it follows that
\begin{equation}\label{eq:2}
  f|_k \tau_d (z)
  =
  \frac{\mu(d)}{d \cdot a_f(d)}
  f(z/d)
  = 
  \frac{\mu(d)}{d \cdot a_f(d)}
  \sum_{n \in \mathbold{N}}
  a_f(n) e((n/d) x) e^{- 2 \pi (n/d) y}.
\end{equation}

\subsubsection*{Fourier expansion of $\theta_{z_1,z_2}$}
Let $d$ be a divisor of $d_B N$.
One can compute the Fourier expansion
of $\theta_{z_1,z_2} |_k \tau_d$
via the adelic Weil representation
(see Remark \ref{rmk:fourier-exp-other-cusps}).
The result obtained is that
\begin{equation}\label{eq:1}
\theta_{z_1,z_2}|_k \tau_d(z)
=
\frac{\mu((d,d_B))}{d}
\sum_{n \in \mathbold{Z}}
U_k(\sigma_{z_1}^{-1} R^{(d)} \sigma_{z_2};
n/d,y) e((n/d) x).
\end{equation}

\begin{proof}[Proof of
  Theorem \ref{thm:zeroth-fourier-coefficient-of-integrand}]
Let $d$ be a divisor of $d_B N$.
By (\ref{eq:2}), (\ref{eq:1}),
and the definition
\eqref{eq:28}
of $c_f(d)$,
the constant
term
of $F(\tau_d z)$
is
\begin{align*}
  a_F(0,y;\tau_d)
  &=
  y^k
  \frac{\mu(d)}{d \cdot a_f(d)}
  \cdot \frac{\mu((d,d_B))}{d}
  \sum_{n \in \mathbold{N}}
  a_f(n)
  U_k(\sigma_{z_1}^{-1} R^{(d)} \sigma_{z_2};n/d,y)
  e^{- 2 \pi (n/d) y}
  \\
  &= 
  \frac{c_f(d)}{d}
  \sum_{n \in \mathbold{N}}
  a_f(n)
  \widetilde{U}_k(\sigma_{z_1}^{-1} R^{(d)} \sigma_{z_2};n/d,y).
\end{align*}
A direct calculation shows that
$a_F(0,y;\tau_d n(j)) =
a_F(0,y;\tau_d)$ for all integers $j$,
so Lemma \ref{lem:coset-repr-1} gives
\begin{align*}
  \sum_{\tau \in \Gamma_0(d_B N) \backslash \SL_2(\mathbold{Z})}
  a_F(0,y;\tau)
  &= 
  \sum_{d | d_B N}
  \sum_{j \in \mathbold{Z}/d}
  a_F(0,y;\tau_d n(j))
  \\
  &=
  \sum_{n \in \mathbold{N}}
  a_f(n)
  \sum_{d|d_B N}
  c_f(d)
  \widetilde{U}_k(\sigma_{z_1}^{-1} R^{(d)} \sigma_{z_2};n/d,y).
\end{align*}
\end{proof}


\subsection{Proofs of formulas stated in \sec\ref{sec:explicit-formula}}\label{sec:comp-shim-lift}
Let notation be as in \sec\ref{sec:overview-1},
and set $F(z) = y^k \overline{f(z)} \theta_{z_1,z_2}(z)$.
Theorem \ref{thm:explicit-shimizu-unadelized} reads
\begin{equation}\label{eq:31}
  (y_1 y_2)^{k/2}
  \overline{f_B(z_1)} f_B(z_2)
  =
  \int _{\Gamma \backslash \mathbold{H} }
  F(z)
  \, \frac{d x \, d y}{y^2}.
\end{equation}
We apply Theorem \ref{thm:compute-integral-of-automorphic-form}
with
$\eps = 100$, $H(s) = s$, and $\delta = 10$, giving
\begin{equation}\label{eq:30}
  \int _{\Gamma \backslash \mathbold{H} }
  F(z)
  \, \frac{d x \, d y}{y^2}
  = \int _{(10)} (2 s - 1) 2 \xi(2 s)
  \sum_{\tau \in \Gamma \backslash \SL_2(\mathbold{Z})}
  a_F(0,\cdot;\tau)^\wedge(1-s) \, \frac{d s}{2 \pi i}.
\end{equation}
Suppose now that $N$ is squarefree.
By
Theorem \ref{thm:zeroth-fourier-coefficient-of-integrand},
we then have
\begin{equation}\label{eq:29}
  \sum_{\tau \in \Gamma \backslash \SL_2(\mathbold{Z})}
  a_F(0,\cdot;\tau)^\wedge(1-s)
  =
  \sum_{n \in \mathbold{N}}
  a_f(n)
  \sum_{d|d_B N}
  c_f(d)
  \widetilde{U}_k(\sigma_{z_1}^{-1} R^{(d)}
  \sigma_{z_2};n/d,\cdot)^\wedge(1-s).
\end{equation}

\begin{lemma}
  \label{lem:integral-Uk-stuff}
  For each lattice $L < M_2(\mathbold{R})$
  and positive real $t$,
  \[
  \int _{(10)} (2 s - 1) 2 \xi (2 s)
  \widetilde{U} _k (L;t,\cdot)^\wedge(1-s)
  \, \frac{d s}{2 \pi i}
  = V_k(L,t).
  \]
\end{lemma}
\begin{proof}[Proof of Theorem \ref{thm:rough-main-2},
  assuming Lemma \ref{lem:integral-Uk-stuff}]
  Combine
  \eqref{eq:31}, \eqref{eq:30}, \eqref{eq:29} and Lemma
  \ref{lem:integral-Uk-stuff}.
\end{proof}
\begin{proof}[Proof of Theorem \ref{thm:ext-maass-deriv}]
  This is obtained
  from Theorem \ref{thm:rough-main-2}
  by differentiating.
  We omit the details, but remark that the resulting
  formula has convincingly passed all numerical tests.
\end{proof}
To prove Lemma \ref{lem:integral-Uk-stuff}, we need
the following technical lemma.
\begin{lemma}
  \label{lem:evaluate-integral-bessel-stuff}
  For each positive real $x$ and $\nu \in \mathbold{C}$
  with $\Re(\nu) \geq 0$,
  \begin{equation*}
    \int _{(10)}
    (s - \tfrac{1}{2})
    \frac{
      \Gamma(s) \Gamma(s+\nu)
    }
    {
      (x/2)^{2s+\nu }
    }
    \, \frac{d s }{2 \pi i}
    = x K_{\nu -1}(x) - K_\nu (x).
  \end{equation*}
\end{lemma}
\begin{proof}
  Mellin inversion
  and the integral formula
  $\int _{0}^\infty 
  x^s K_\nu(2 x) \, d^\times x
  = 
  \frac{1}{4}
  \Gamma \left( \frac{s + \nu }{2} \right)
  \Gamma \left( \frac{s -  \nu }{2} \right)$
  valid for $\Re(s \pm \nu)  > 0$
  (see \cite[6.561.16]{GR})
  show that for $\Re(\nu) > -5$,
  \begin{align*}
    K_\nu(x)
    &= \frac{1}{4}
    \int _{(\Re(\nu) + 5)}
    \frac{
      \Gamma \left( \frac{s + \nu }{2} \right)
      \Gamma \left( \frac{s -  \nu }{2} \right)
    }{
      (x/2)^s
    }
    \, \frac{d s}{2 \pi i}
    = \frac{1}{2}
    \int _{(10)}
    \frac{
      \Gamma \left( s\right)
      \Gamma \left( s + \nu \right)
    }{
      (x/2)^{2s+\nu}
    }
    \, \frac{d s}{2 \pi i}.
  \end{align*}
  For $\Re(\nu) > -4$,
  the identity $s \Gamma(s) = \Gamma(s+1)$
  gives
  \[
  x K_{\nu-1}(x)
  =
  \int _{(9)}
  \frac{\Gamma(s) \Gamma(s+\nu-1)}{
    (x/2)^{2s + \nu-2}}
  \, \frac{d s}{2 \pi i}
  =
  \int _{(10)}
  s
  \frac{\Gamma(s) \Gamma(s+\nu)}{
    (x/2)^{2s + \nu}}
  \, \frac{d s}{2 \pi i}.
  \]
  Subtracting, we are done.
\end{proof}
\begin{proof}[Proof of Lemma \ref{lem:integral-Uk-stuff}]
  Since $P(\alpha) + \det(\alpha) = 2 |X(\alpha)|^2$,
  we see that
  \begin{align*}
    \widetilde{U}_k(L;t,y)
    &= \frac{y ^{k + 1}}{2} \sum _{\substack{
        \alpha \in L  \\
        \det (\alpha) = t
      }
    }
    X (\alpha) ^k e ^{- 2 \pi (P (\alpha) + t) y}
    = \frac{y ^{k + 1}}{2} \sum _{\substack{
        \alpha \in L  \\
        \det (\alpha) = t
      }
    }
    X (\alpha) ^k e ^{- 4 \pi |X(\alpha)|^2 y}.
  \end{align*}
  Thus for each $s \in \mathbold{C}$ with $\Re(s) = 10$,
  \begin{align}\label{eq:32}
    \nonumber
    \widetilde{U} _k (L;t,\cdot)^\wedge(1-s)
    &=
    \nonumber
    \int _{y \in \mathbold{R} _+ ^\times }
    \widetilde{U} _k (L;t,y)
    y ^{s - 1 }
    \, d^\times y
    \\
    \nonumber
    &=
    \frac{1}{2}
    \sum _{
      \substack{
        \alpha \in L \\
        \det(\alpha) = t
      }
    }
    X(\alpha)^k
    \frac{\Gamma(s + k)}{(4 \pi |X(\alpha)|^2)^{s+k}} \\
    &=
    \sum _{
      \substack{
        \alpha \in L \\
        \det(\alpha) = t
      }
    }
    e^{i k \arg X(\alpha)}
    Z_k(4 \pi \lvert X(\alpha) \rvert; s),
  \end{align}
  where
  $Z_k(x;s)
  :=
  2^{-k-1}
  \pi^s
  \Gamma(s + k)
  (x/2)^{-2 s-k}$;
  here the interchange of summation and integration
  is justified by absolute convergence.
  Since
  $\xi(2 s) \pi^s = \Gamma(s) \sum_{m \in \mathbold{N}} m^{- 2
    s}$,
  it follows from
  Lemma \ref{lem:evaluate-integral-bessel-stuff}
  that
  \begin{align}\label{eq:35}
    \int _{(10)}
    (2 s - 1) 2 \xi(2 s)
    Z_k(x;s)
    \, \frac{d s}{2 \pi i}
    &=
    2^{1-k}
    \sum_{m \in \mathbold{N}}
      m^k
    \int _{(10)}
    (s - \tfrac{1}{2})
    \frac{
      \Gamma(s) \Gamma(s+k)
    }
    {
      (m x/2)^{2s+k}
    }
    \, \frac{d s }{2 \pi i} \\
    \nonumber
    &=
    2^{1-k}
    \sum_{m \in \mathbold{N}}
    m^k
    ( m xK_{k-1}(m x) - K_k(m x)) \\ \nonumber
    &=
    W_k(x).
  \end{align}
  Combining \eqref{eq:32} and \eqref{eq:35} (with $x = 4 \pi |X(\alpha)|$),
  we conclude that
  \[
  \int _{(10)} (2 s - 1) 2 \xi (2 s)
  \widetilde{U} _k (L;t,\cdot)^\wedge(1-s)
  \, \frac{d s}{2 \pi i}
  =
  \sum _{
    \substack{
      \alpha \in L \\
      \det(\alpha) = t
    }
  }
  e^{i k \arg X(\alpha)}
  W_k(4 \pi |X (\alpha)|),
  \]
  which equals
  $V_k(L,t)$ by definition.
\end{proof}

\section{Explicit Shimizu  correspondence}
\label{sec:shim-corr-1}
In this appendix, we
explain how Theorem
\ref{thm:explicit-shimizu-unadelized}
follows from
the Shimizu correspondence
as pinned down by Watson \cite[Thm 1]{watson-2008}.

\subsection{Notation, measures}\label{sec:notation-measures}
Let $B$ be an
indefinite rational quaternion algebra.
Recall the
notation from \secs\ref{sec:quaternion-algebras2}
and \ref{sec:newf-shim-curv}.
The discussion that follows applies to non-split $B$ as well as to
$B = M_2(\mathbold{Q})$.
We regard $B$, $B^\times$,
$P B^\times$,
$B^1 := \{b \in B^\times : \det(b) = 1\}$,
and $Z(B^\times) = $ center of $B^\times$
as algebraic groups over $\mathbold{Q}$.

Let $\mathbold{A} = \prod ' \mathbold{Q}_v
= \mathbold{R} \times \hat{\mathbold{Q}}$
($\hat{\mathbold{Q}} = \prod_p' \mathbold{Q}_p$)
be the adele ring of $\mathbold{Q}$.
For $G$ an algebraic group over $\mathbold{Q}$,
write $G_v = G(\mathbold{Q}_v)$.
We identify
$G(\mathbold{A})= G_\infty \times
G(\hat{\mathbold{Q}})$
without mention.

Let $G_{\infty+}$
be the topologically connected component
of $G_\infty$,
which for $G = B^\times$ or $PB^\times$ is
the subgroup of positive determinant elements.
We fix
compatible identifications
$B_\infty = M_2(\mathbold{R})$
and $B_\infty^\times  = \GL_2(\mathbold{R})$,
and let $K_\infty = \SO(2) \subset B_{\infty+}^\times$
be the maximal compact connected subgroup
of
$B^\times_{\infty+}$
stabilizing $i \in \mathbold{H}$.

Let $X = $ either $\mathbold{Q}_v$ or $B_v$.
Define
$\langle , \rangle : X \times X \rightarrow
\mathbold{Q}_v$
by $\mathbold{Q}_v \times \mathbold{Q}_v \ni (x,y)
\mapsto x y$
and
$B_v \times B_v \ni (\alpha,\beta) \mapsto
\tr(\alpha \beta^{\iota})
= \alpha \beta^{\iota} + \beta \alpha ^{\iota}$.
Let
$\mathcal{S}(X)$
denote the space of Schwarz-Bruhat
functions
on $X$.
Let $\mathbf{e} = \prod \mathbf{e} _v  \in
\Hom(\mathbold{A}/\mathbold{Q},S^1)$
be the standard additive character, characterized
by requiring
$\mathbf{e}_\infty(x) = e(x) := e^{2 \pi i x}$.
Let
$d x$ be the Haar measure on $X$
for which the Fourier transform
$\mathcal{F} : \mathcal{S}(X) \rightarrow \mathcal{S}(X)$
given by
$\mathcal{F} \varphi(y) :=
\int _{B_v} \varphi(x) e_v (\langle x, y
\rangle) \, d x$
satisfies $\mathcal{F} \mathcal{F} \varphi(x) = \varphi(-x)$.
For $x \in X$, let $|x|$ be its modulus,
so that $[y \mapsto x y]^* d y = |x| \, d y$.
Let $d^\times x$ be the measure $\zeta_v(1) |x|^{-1} \, d x$
on $X^\times$,
where $\zeta_p(s) := (1-p^{-s})^{-1}$ and $\zeta_\infty(s)
:= \pi^{-s/2} \Gamma(s/2)$.

The measures so-defined on $\mathbold{Q}_v^\times$
and $B_v^\times$
induce measures $d^\times \alpha$
on $PB_v^\times$ and $d^{(1)} \alpha$ on $B_v^1$
via the short exact sequences
$1 \rightarrow B_v^1 \rightarrow B_v^\times \xrightarrow{\nu}
\mathbold{Q}_v^\times \rightarrow 1$
and 
$1 \rightarrow \mathbold{Q}_v^\times \rightarrow B_v^\times \rightarrow
PB_v^\times \rightarrow 1$.
This normalization is consistent with
\cite[2.1.2]{watson-2008}.
For $p \nmid d_B$ (resp. $p \mid d_B$),
the measure $d^{(1)} \alpha$ assigns volume
$\zeta_p(2)^{-1}$ (resp. $(p-1)^{-1} \zeta_p(2)^{-1}$)
to the group of norm one units
in a maximal order in $B_p$.
Let $\sigma_z$ ($z \in \mathbold{H}$)
be as in \sec\ref{sec:main-formula}.
Then on $B_\infty^1$, the measure
$d^{(1)} \alpha$ is given by
\[
d^{(1)} \alpha =
\frac{1}{2}
d x \, \frac{d ^\times y}{|y|} \, d \theta
\quad \text{ if }
\alpha =
\sigma_{x+iy}
 \begin{bmatrix}
  \cos(\theta) & \sin(\theta) \\
  -\sin(\theta)& \cos(\theta)
\end{bmatrix}.
\]

The product
$d^{(1)} \alpha :=  \prod_v d^{(1)} \alpha_v$
of measures on $B^1_v$
converges
to a measure on $B^{1}(\mathbold{A})$,
inducing a quotient measure on
$[B^1] := B^1(\mathbold{Q}) \backslash B^1(\mathbold{A})$
by giving $B^1(\mathbold{Q})$
the counting measure.
We may define similarly
a product measure
$d^\times  \alpha :=  \prod_v d^\times \alpha_v$
 on $P B^\times(\mathbold{A})$
 inducing a quotient measure on
 $[PB^\times] := PB^\times(\mathbold{Q}) \backslash
PB^\times(\mathbold{A})$.

\subsection{Adelization}\label{sec:forw-reverse-adel}
Let
$R \subset B \hookrightarrow M_2(\mathbold{R})$
be an Eichler order and real embedding
giving rise to the lattice $\Gamma
= R \cap \SL_2(\mathbold{R}) < \SL_2(\mathbold{R})$.
Let $R_p = R \otimes_\mathbold{Z} \mathbold{Z}_p$
and $\hat{R} = \prod R_p \subset B(\hat{\mathbold{Q}})$.
If $\mathcal{F}(\Gamma \backslash \mathbold{H})$
is a fundamental domain for $\Gamma \backslash \mathbold{H}$,
then
\[
\left\{
  \sigma_z \left(
    \begin{smallmatrix}
      \cos \theta & \sin \theta \\
      - \sin \theta & \cos \theta 
    \end{smallmatrix}
  \right) \times \kappa_0 :
  z \in \mathcal{F}(\Gamma \backslash \mathbold{H}),
  \,
  0 \leq \theta < \pi,
  \,
  \kappa_0 \in \hat{R}^\times
\right\}
\]
is  a
fundamental domain for $[\SL_2]$.
Fix a positive even integer $k$.
The following lemma
is a consequence the strong approximation theorem.
\begin{lemma}\label{lem:adelize}
  The identity
  $F(\sigma_z \times 1) = y^{k/2} f(z)$
  for $z = x + i y \in \mathbold{H}$
  induces a bijection
  $f \leftrightarrow F$
  between the space of automorphic
  functions $f$  of weight $k$ on $\Gamma$ 
  and the space of smooth
  functions $F : B^\times_\mathbold{A} \rightarrow
  \mathbold{C}$
  satisfying
  $F(z \gamma g \kappa_\infty \kappa_0)
  = F(g) \chi_k(\kappa_\infty)$,
  where
  $z \in Z(B^\times)_\mathbold{A}$, $\gamma \in
  B^\times_\mathbold{Q}$, $g \in B^\times_\mathbold{A}$,
  $\kappa_\infty \in K_\infty \subset B_{\infty+}^\times$,
  $\kappa_0 \in \hat{R}^\times$
  and  $\chi_k : K_\infty = \SO(2) \rightarrow S^1$
  is given by
  $\left(
    \begin{smallmatrix}
      \cos \theta & \sin \theta \\
      - \sin \theta & \cos \theta 
    \end{smallmatrix}
  \right) \mapsto e^{i k \theta}$.
  If $f_1 \leftrightarrow F_1$
  and
  $f_2 \leftrightarrow F_2$,
  then
\begin{align*}
\int_{[B^1]}
\overline{F_1(g)} F_2(g) \, d^{(1)} g
&=
\frac{1}{2}
\int_{[P B^\times]}
\overline{F_1(g)} F_2(g) \, d^{\times} g \\
&=
\frac{1}{c(\Gamma)}
\int _{\Gamma \backslash \mathbold{H}}
y^k \overline{f_1(z)}
f_2(z)
\, \frac{d x \, d y}{y^2},
\end{align*}
where
$c(\Gamma)
:= 2 \pi^{-1} \zeta(2) d_B N \prod_{p \mid d_B} (1 - p^{-1})
\prod_{p \mid N} (1 + p^{-1})$.
\end{lemma}
\begin{remark}
\label{rmk:dfdfd}
One has $c(\Gamma) = \vol(\Gamma \backslash \mathbold{H})$
(see Remark \ref{rmk:volumes-known}),
or equivalently,
$\vol([B^1])
=
(1/2) \vol([PB^\times])
=1$.
\end{remark}



\subsection{Weil representation}
\label{sec:construction-general}
Let $B$ be a quaternion algebra over $\mathbold{Q}$.
Let $\GO(B)$ be the orthogonal similitude group of
the four-dimensional quadratic space
$(B,\det)$,
and $\nu : \GO(B) \rightarrow \mathbold{G}_m$ the similitude factor.
Let $H = G(\SL_2 \times \O(B)) = \{(g,h) \in \GL_2 \times \GO(B)
: \det(g) = \nu(h)\}$.
Then $H(\mathbold{A})$
acts on the space
$\mathcal{S}(B(\mathbold{A}))= \otimes'
\mathcal{S}(B_v)$
of Schwarz-Bruhat functions
via the Weil representation $\omega = \otimes \omega_v$,
characterized by
(see \cite[\S 3]{harris-kudla-1991})
\begin{align*}
  \omega_v (n(x) t(y),1)
  \varphi_v
  (\alpha)
  &= |t|_v^{4/2}
  \mathbf{e}_v(x \det(\alpha))
  \varphi_v(t \alpha) \quad \text{for }
  x \in \mathbold{Q}_v, y \in \mathbold{Q}_v^*,
  \\
  \omega_v (w,1)
  \varphi_v(\alpha) &= (-1)^{\mathbf{1}_{v|d_B}}
  \mathcal{F} \varphi_v(\alpha),  \\
  \omega_v
  \left(
    \begin{bmatrix}
      1 &  \\
      & \nu(h)
    \end{bmatrix},
    h
  \right)
  \varphi_v(\alpha)
  &= \lvert \nu(h) \rvert_v^{-4/4} \varphi_v(h^{-1} \alpha)
  \quad \text{ for } h \in \GO(B)_v.
\end{align*}
Here $\mathbf{1}_{v|d_B}$ is $1$ if $v$ divides $d_B$ and $0$
otherwise,
while
\begin{equation*}\label{eq:matrix-notation}
  n(x) := \begin{bmatrix}
    1 & x \\
    & 1
  \end{bmatrix},
  \quad
  t(y) :=
  \begin{bmatrix}
    y &  \\
    & y ^{-1} 
  \end{bmatrix},
  \quad
  w := \begin{bmatrix}
    & 1 \\
    -1 & 
  \end{bmatrix}.
\end{equation*}

\subsection{Theta correspondence}
\label{sec:shimizu-lift-subsec}
There is a morphism
$\rho : B^\times \times B^\times \rightarrow \GO(B)$
given by
$\rho(b_1,b_2)(\alpha) = b_1 \alpha b_2^{-1}$,
which surjects onto the connected component
and satisfies
$\det(b_1 b_2^{-1}) = \nu(\rho(b_1,b_2))$.
For $\varphi \in \mathcal{S}(B_\mathbold{A})$,
$g \in \GL_2(\mathbold{A})$
and $(b_1,b_2) \in (B^\times  \times B^\times)(\mathbold{A})$
with $\det(g) = \det(b_1 b_2^{-1})$,
define the theta kernel
\begin{equation*}
  \theta_\varphi(g;b_1,b_2)
  = \sum_{\alpha \in B}
  \omega(g,\rho(b_1,b_2)) \varphi(\alpha).
\end{equation*}
For fixed $b_1, b_2 \in B^\times (\mathbold{A})$
and $g' \in \GL_2(\mathbold{A})$ with $\det(g') = \det(b_1
b_2^{-1})$,
the function $\SL_2(\mathbold{A}) \ni g \mapsto \theta_\varphi(g
g')$
is left-$\SL_2(\mathbold{Q})$-invariant.
If
$F$ is a function on $\GL_2(\mathbold{Q}) \backslash
\GL_2(\mathbold{A})$
of rapid decay modulo the center,
define for $b_1,b_2 \in B^\times(\mathbold{A})$ its theta lift
\[
\Theta_\varphi(\bar{F})(b_1,b_2)
= \int_{g \in [\SL_2]}
\bar{F}(g g')
\theta_\varphi(g g';b_1,b_2) \, d^{(1)} g,
\]
where $g' \in \GL_2(\mathbold{A})$ is any element
for which $\det(g') = \det(b_1 b_2^{-1})$.

\subsection{The precise lifting for newforms}
\label{sec:prec-corr-newforms-1}
Let
notation
be as in \sec\ref{sec:overview-1}.
Let $f_B \leftrightarrow F_B$ and $f \leftrightarrow F$
via
Lemma \ref{lem:adelize}.
Define $\varphi = \otimes \varphi_v \in
\mathcal{S}(B(\mathbold{A}))$ via
$\varphi_p
= 
\vol(R_p^\times, d^\times x)^{-1}
\mathbf{1}_{R_p}$
and
$\varphi_\infty = \pi^{-1} X^k e^{- 2 \pi P}$,
where $\mathbf{1}_{R_p}$ is the characteristic function
of $R_p$,
and $X$ and $P$ are
as
in \sec\ref{sec:weighted-sums-over}.
This choice is consistent with
\cite[\S2.3]{watson-2008}.
The definitions imply:


\begin{lemma}\label{lem:theta-vs-theta}
  For $z, z_1, z_2 \in \mathbold{H}$,
  we have
  $\theta_{\varphi}(\sigma_z;\sigma_{z_1},\sigma_{z_2})
  =
  c(\Gamma ')
  y^{k/2}
  \theta_{z_1,z_2}(z).$
\end{lemma}
\begin{remark}\label{rmk:fourier-exp-other-cusps}
  To get the Fourier expansion for $\theta_{z_1,z_2}|\tau_d(z)$
  claimed in Appendix \ref{sec:fourier-expansions},
  we use that
  $c(\Gamma ') y^{k/2} \theta_{z_1, z_2} | \tau_d(z)
  = \theta_\varphi(\tau_d \sigma_z \times
  1;\sigma_{z_1},\sigma_{z_2})
  = \theta_\varphi(\sigma_z \times
  \tau_d^{-1};\sigma_{z_1},\sigma_{z_2})$
  and apply the definition of the Weil representation
  at the finite places.
  The key calculation is that
  at a prime $p$ for which
  we may identify $B_p = M_2(\mathbold{Q}_p)$,
  we have
  \[
  \mathcal{F} 1 _{\left(
      \begin{smallmatrix}
        \mathbold{Z}_p& \mathbold{Z}_p\\
        p^\alpha  \mathbold{Z}_p &\mathbold{Z}_p
      \end{smallmatrix}
    \right)}
  = p^{-\alpha}
  1 _{\left(
      \begin{smallmatrix}
        \mathbold{Z}_p& p^{-\alpha} \mathbold{Z}_p\\
        \mathbold{Z}_p &\mathbold{Z}_p
      \end{smallmatrix}
    \right)}.
  \]
  We omit the details, since they are discussed
  in \cite{nelson-appendix-C-prasanna-notes} and \cite[\S 2.1]{watson-2008}.
\end{remark}

\begin{theorem}
  \label{thm:watson-1}
  For $b_1, b_2 \in B^\times(\mathbold{A})$,
  we have
  $\overline{F _B}(b_1) F_B(b_2)
  =
  \frac{\|F_B\|^2}{\|F\|^2}
  \Theta_\varphi(\bar{F})(b_1,b_2)$,
  where
  \[
  \|F_B\|^2 :=
  \frac{1}{2}
  \int _{
    [PB^\times]
  }
  |F_B|^2
  (g)
  \, d^\times g,
  \quad 
  \|F\|^2 :=
  \frac{1}{2}
  \int _{
    [\PGL_2]
  }
  |F|^2
  (g)
  \, d^\times g.\]
\end{theorem}
\begin{proof}
  This is \cite[Thm 1]{watson-2008}.
  Although
  Watson assumes throughout his paper that $N$ is squarefree,
  his proof of this particular result applies verbatim for general $N$.
\end{proof}

\begin{proof}[Proof of Theorem
  \ref{thm:explicit-shimizu-unadelized}]
  By Lemma \ref{lem:adelize},
  Lemma \ref{lem:theta-vs-theta}
  and
  Theorem
  \ref{thm:watson-1},
  \begin{align*}
    (y_1 y_2)^k
    \overline{f_B(z_1)}
    f_B(z_2)
    &=
    \overline{F_B(\sigma_{z_1})}
    F_B(\sigma_{z_2})
    = \Theta_{\varphi}(\overline{F})(\sigma_{z_1},\sigma_{z_2})
    \\
    &=
    \frac{\|F_B\|^2}{ \|F\|^2}
    \int _{[\SL_2]}
    \overline{F}(g)
    \theta_\varphi(g;\sigma_{z_1},\sigma_{z_2}) \, d^{(1)} g \\
    &=
    \frac{\|F_B\|^2}{ \|F\|^2}
    \frac{1}{c(\Gamma)}
    \int _{\Gamma \backslash \mathbold{H}}
    y^{k/2} \overline{f(z)}
    \theta_\varphi(\sigma_z;\sigma_{z_1},\sigma_{z_2}) \,
    \frac{d x \, d y}{y^2} \\
    &=
    C
    \int _{\Gamma \backslash \mathbold{H}}
    y^{k} \overline{f(z)}
    \theta_{z_1,z_2}(z) \,
    \frac{d x \, d y}{y^2},
    \quad
    C := \frac{\|F_B\|^2}{ \|F\|^2}
    \frac{
      c(\Gamma ')
    }{
      c(\Gamma)
    }.
    \\
  \end{align*}
  Since $f_B$ and $f$ are compatibly-normalized
  (see Definition \ref{defn:compatibly-normalized}),
  we have $C = 1$.
\end{proof}

\bibliography{refs}{}
\bibliographystyle{plain}
\end{document}